\documentclass[11pt,a4paper,twoside]{amsart}
\usepackage{amssymb,amsmath,amsthm}
\theoremstyle{plain}
\newtheorem{theorem}{Theorem}[section]

\newtheorem{lemma}[theorem]{Lemma}

\newtheorem{proposition}[theorem]{Proposition}
\newtheorem{assumption}[theorem]{Assumption}
\theoremstyle{remark}
\newtheorem{remark}[theorem]{Remark}
\newtheorem{notation}[theorem]{Notation}
\newtheorem{example}[theorem]{Example}

\usepackage[dvipdfm,
  bookmarks=true,
  bookmarksnumbered=false,
  bookmarkstype=toc]{hyperref}
\numberwithin{equation}{section}

\newcommand{\R}{\mathbb{R}}

\renewcommand{\Re}{\operatorname{Re}}
\newcommand{\I}{\infty}

\newcommand{\norm}[1]{\left\lVert #1\right\rVert}
\newcommand{\Lebn}[2]{\left\lVert #1 \right\rVert_{L^{#2}}}

\newcommand{\tnorm}[1]{\lVert #1\rVert}

\newcommand{\tSobn}[2]{\lVert #1 \rVert_{H^{#2}}}

\newcommand{\IN}{\quad\text{in }}

\def\({\left(}
\def\){\right)}
\def\<{\left\langle}
\def\>{\right\rangle}
\def\le{\leqslant}
\def\ge{\geqslant}

\def\d{{\partial}}
\def\l{\lambda}
\newcommand{\al}{\alpha}

\newcommand{\g}{\gamma}

\newcommand{\eps}{\varepsilon}

\newcommand{\ti}{\widetilde}

\newcommand{\AG}{{\frac{\alpha}{\gamma-\alpha}}}

\begin{document}
\title[Cascade of phase shifts for Hartree equation]{Cascade of phase shifts
and creation of nonlinear focal points
for supercritical Semiclassical Hartree equation}
\author[S. Masaki]{Satoshi Masaki}
\address{Division of Mathematics\\
Graduate School of Information Sciences\\
Tohoku University\\
Sendai 980-8579, Japan}
\email{masaki@ims.is.tohoku.ac.jp}
\begin{abstract}
We consider the semiclassical limit of the Hartree equation
with a data causing a focusing at a point.
We study the asymptotic behavior of phase function
associated with the WKB approximation
near the caustic when a nonlinearity is supercritical.
In this case, it is known that a phase shift occurs in a neighborhood of focusing time
in the case of focusing cubic nonlinear Schr\"odinger equation.
Thanks to the smoothness of the nonlocal nonlinearities,
we justify the WKB-type approximation of the solution for a
data which is larger than in the previous results
and is not necessarily well-prepared.
We also show by an analysis of the limit hydrodynamical equaiton that, however,
this WKB-type approximation breaks down before reaching the focal point:
Nonlinear effects lead to the formation of singularity
of the leading term of the phase function.
\end{abstract}
\maketitle

\section{Introduction}\label{sec:intro}
This paper is devoted to the study of the semiclassical limit $\eps \to 0$
for the Cauchy problem of the semiclassical nonlinear Schr\"odinger equation
for $(t,x) \in \R_+ \times \R^n$
\begin{align}\label{eq:r3}
	i\eps\d_t u^\eps +\frac{\eps^2}{2}\Delta u^\eps =& \eps^\al N(u^\eps),&
	u^\eps_{|t=0}(x) =& a_0(x)e^{-i\frac{|x|^2}{2\eps}}
\end{align}
with the nonlocal nonlinearity of Hartree type
\begin{equation}\label{eq:h}
	N(u^\eps) = \lambda (\lvert x\rvert ^{-\gamma}\ast  |u^\eps|^2)u^\eps,
\end{equation}
where $n\ge 3$, $\l$ is a real number, and $\g$ is a positive number.
In this paper we consider the small-nonlinearity case $\alpha>0$.
The case $\al=0$ is studied in \cite{AC-SP,CM-AA}.
In the case of linear equation $N \equiv 0$,
the quadratic oscillation in the initial data causes a caustic at the origin at $t=1$.
In \cite{CaIUMJ,CaCMP}, R. Carles justified the general heuristics presented in \cite{HK-WM}
in the case of \eqref{eq:r3} with $N(y) = |y|^\beta y$ (see, also \cite{CaBook}),
and two different notions of the criticality for $\al$ are realized.
One is concerned with the nonlinear effect on the behavior far from the focal point,
and the other with that near the focal point.
The situation is similar in the case of Hartree equation
\begin{align}\label{eq:r4}
  i\eps\d_t u^\eps +\frac{\eps^2}{2}\Delta u^\eps &= \lambda
                \eps^\al (\lvert x\rvert ^{-\gamma}\ast  |u^\eps|^2)u^\eps, \\
	\label{eq:odata}
	u^\eps_{|t=0}(x) &= a_0(x)e^{-i\frac{|x|^2}{2\eps}}.
\end{align}
(see, \cite{CL-FT,CMS-SIAM,MaASPM}).
The two critical indices are $\al = 1$ (far from focal point) and
$\al = \g$ (near the focal point).
They are completely different notions.
For example, the second index $\al=\g$ depends on the shape
of the nonlinearity while the first not.
Using terminologies in \cite{CaJHDE}, we have the following nine cases:
\begin{center}
\begin{tabular}{c|c|c|c|}
 & $\alpha > 1$ & $\alpha = 1$ & $\alpha < 1$ \\
\hline
$\alpha>\gamma$ & linear WKB & nonlinear WKB & supercritical WKB \\
 & linear caustic & linear caustic & linear caustic \\
\hline
$\alpha=\gamma$ & linear WKB & nonlinear WKB & supercritical WKB \\
 & nonlinear caustic & nonlinear caustic & nonlinear caustic  \\
\hline
$\alpha<\gamma$ & linear WKB & nonlinear WKB & supercritical WKB \\
 & supercritical caustic & supercritical caustic & supercritical caustic  \\
\hline
\end{tabular}
\end{center}
Roughly speaking, the term ``linear WKB'' means that, far from the caustic,
the propagation of $u^\eps$ does not involve the nonlinear effect
at leading order.
The term ``linear caustic'' means that the nonlinear effect is negligible at leading order
when the solution crosses the focal point. 

Now, let us be more precise about this problem with \eqref{eq:r4}--\eqref{eq:odata}.
Let $w^\eps $ be the solution of the linear equation
$(i\eps \partial_t + (\eps^2/2)\Delta)w^\eps=0$
with the same initial data as in \eqref{eq:odata}.
Note that $w^\eps$ is approximated by the WKB-type approximate solution
\begin{equation}\label{eq:linear}
	v^\eps_{\mathrm{lin}}=\frac{1}{(1-t)^{n/2}}a_0 \(\frac{x}{1-t}\) e^{i\frac{|x|^2}{2\eps (t-1)}}
\end{equation}
as long as $\eps /(1-t)$ is small.
In the linear WKB case $\al>1$,
it is shown in \cite{CL-FT,MaAHP,MaASPM} that
the solution $u^\eps$ is close to $w^\eps$ (and so to $v^\eps_{\mathrm{lin}}$)
before caustic.
The time interval in which $u^\eps$ is approximated by $w^\eps$ 
depends on the latter critical notion $\al=\g$:
If $\al>\g$ then $u^\eps \to w^\eps$ as $\eps\to0$ for all $t \le 1$;
if $\al=\g$ then it holds for $1- t \ge  \Lambda \eps$ with a large $\Lambda$;
if $\al<\g$ then it holds for $1- t \ge  \Lambda \eps^\mu$ with a large $\Lambda$
and some $\mu=\mu(\al,\g) \le (\al-1)/(\g-1)$.
In the case $\al \ge \g$,
the asymptotic profile of the solution beyond the caustic is also given in \cite{CL-FT}.
Let us proceed to the supercritical caustic case $\al<\g$.
In this case, infinitely many phase shifts occur between the time $1- t = \Lambda_1\eps^{\frac{\al-1}{\g-1}}$,
called a \emph{first boundary layer}, and the time
$ 1-t = \Lambda \eps^{{\al}/{\g}}$,
called a \emph{final layer}.
This phenomena, called \emph{cascade of phase shifts},
is first shown in \cite{CaJHDE} for \eqref{eq:r3} with 
certain class of nonlinearities including the cubic nonlinearity $N(y) = |y|^2 y$,
and the asymptotic behavior of the solution is given
before the final layer, $1-t\gg \eps^{{\al-1}/{\g-1}}$,
 in the case $(\gamma>)\alpha>1$.
The similar result is proven also
in the nonlinear and supercritical WKB case $\alpha \le 1$,
provided the initial data is a properly-modified one of the form
\begin{equation}\label{eq:wpdata}
	u_{|t=0}^\eps(x) = b_0(\eps^{\frac{\al}{\g}},x) e^{-i\frac{|x|^2}{2\eps}} \exp (i\eps^{\frac{\al}{\g}-1}\phi_0(\eps^{\frac{\al}{\g}},x)),
\end{equation}
where $(b_0(t,x),\phi_0(t,x))$ is a suitable function defined in terms of  $a_0$.
Let us call this type of data as a \emph{well-prepared data}.

The aim of this paper is 
to give an explicit asymptotic profile of the solution of \eqref{eq:r4}-\eqref{eq:odata}
(before and) on the final layer for whole $\g > \al >0$ with a not-modified data.
In particular, the behavior of the solution is new in the following situations:
\begin{itemize}
\item On the final layer,
that is, at $t=1-T^{-1} \eps^{\al/\g}$ with not necessarily small $T>0$.
\item $\al \le 1$ with an initial data which is not necessarily of the form \eqref{eq:wpdata}.
\end{itemize}
Moreover, it will be shown that the WKB-type approximation 
of the solution breaks down at some time on the final layer, that is, at $t=t_c:=1-(T^{*})^{-1} \eps^{\al/\g}$
for some $T^*> 0$ for a certain class of initial data.

For our analysis, we apply the following transform introduced in \cite{CaJHDE}:
\begin{equation}\label{eq:spct}
	u^\eps(t,x) = \frac{1}{(1-t)^{n/2}} \psi^\eps \( \frac{\eps^{\frac{\al}{\g}}}{1-t}, \frac{x}{1-t} \)
	\exp\(i \frac{|x|^2}{2\eps(t-1)}\).
\end{equation}
Then, \eqref{eq:r4}--\eqref{eq:odata} becomes
\begin{align*}
  i\eps^{1-\frac{\al}{\g}} \d_\tau \psi^\eps +\frac{(\eps^{1-\frac{\al}{\g}})^2}{2}\Delta_y \psi^\eps &= \lambda
                \tau^{\g-2} (\lvert y\rvert ^{-\g}\ast  |\psi^\eps|^2)\psi^\eps, &
	\psi^\eps_{|\tau=\eps^{\al/\g}}(y) &= a_0(y),
\end{align*}
where $\tau=\eps^{\al/\g}/(1-t)$ and $y=x/(1-t)$.
The quadratic oscillation of the initial data is canceled out. 
This transformation clarifies the problem;
if the nonlinear effect is weak and so if $u^\eps$ behaves like $w^\eps$
and $v^\eps_{\mathrm{lin}}$, then $\psi^\eps$ is close to $a_0$;
if the solution has a rapid oscillation other than $\exp(i|x|^2/2\eps(t-1))$,
then $\psi^\eps$ becomes oscillatory.
We change the parameter into $h=\eps^{1-\al/\g}$.
The limit $\eps\to 0$ is equivalent to $h\to 0$ as long as $\al<\g$.
Denoting $\psi^\eps$ by $\psi^h$, 
our problem is reduced to the limit $h\to 0$ of 
the solution to 
\begin{align}\label{eq:r5}
  ih\d_\tau \psi^h +\frac{h^2}{2}\Delta_y \psi^h &= \l
                \tau^{\g-2} (\lvert y\rvert ^{-\g}\ast  |\psi^h|^2)\psi^h, &
	\psi^h_{|\tau=h^{\frac{\al}{\g-\al}}}(y) &= a_0(y).
\end{align}
Since $\tau=h^{\al/(\g-\al)}/(1-t)$,
the correspondence between boundary layers of $t$ and $\tau$ variables is as follows:
\begin{equation}\label{eq:timetranslation}
\begin{aligned}
\text{initial time:}& & t&{}=0 &&\longleftrightarrow 
 \tau = h^{\frac{\al}{\g-\al}}, 	\\
\text{first layer:}& & t&{}=1-\Lambda_1\eps^{\frac{\al-1}{\g-1}}&&
 \longleftrightarrow  \tau = \Lambda_1^{-1} h^{\frac{1}{\g-1}}, 	\\
\text{final layer:}& & t&{}=1-T^{-1}\eps^{\frac{\al}{\g}}&&
 \longleftrightarrow  \tau = T.
\end{aligned}
\end{equation}

Our analysis of \eqref{eq:r5} is based on a generalized WKB method by G\'erard \cite{PGEP} and Grenier \cite{Grenier98}.
We apply a modified Madelung transformation
\begin{equation}\label{eq:grenier}
	\psi^h = a^h \exp\(i\frac{\phi^h}h\)
\end{equation}
to \eqref{eq:r5}, where $a^h$ is complex valued and $\phi^h$ is real valued.
This choice is slightly different from the usual Madelung transformation
\[
	\psi^h = \sqrt{\rho^h} \exp\(i\frac{S^h}h\)
\]
(see \cite{GLM-TJM}) which leads us to an equation of
compressible fluid with the quantum pressure.
It is essential that $a^h$ takes complex value and, therefore, $\phi^h \neq S^h$ in general.
The choice \eqref{eq:grenier} allows us to rewrite \eqref{eq:r5} as the system for the pair $(a^h,\phi^h)$:
\begin{equation}\label{eq:sys}
	\left\{
	\begin{aligned}
		&\partial_\tau a^h + \nabla \phi^h \cdot \nabla a^h + \frac12 a^h \Delta \phi^h = i\frac{h}{2}\Delta a^h,\\
		&\partial_\tau \phi^h + \frac12 |\nabla \phi^h|^2 + \l \tau^{\g-2}(|y|^{-\g}*|a^h|^2) =0, \\
		& a^h_{|\tau=h^{\frac{\al}{\g-\al}}}=a_0, \quad \phi^h_{|\tau=h^{\frac{\al}{\g-\al}}}=0.
	\end{aligned}
	\right.
\end{equation}
If we approximate the solution of \eqref{eq:sys} up to $h^1$ order,
that is, if we establish an asymptotics such as
\begin{align}\label{eq:h1expansion}
	a^h&{}=b_0+h b_1 + o(h^1), &
	\phi^h&{}=\phi_0 + h \phi_1 + o(h^1),
\end{align}
then, by means of \eqref{eq:grenier},
we immediately obtain
the WKB-type  approximate solution  of $\psi^h$.
This method is first employed for nonlinear Schr\"odinger equation with 
a certain class of defocussing local nonlinearity including the cubic nonlinearity
$N(y)=|y|^2 y$ for analytic data \cite{PGEP}
and for Sobolev data \cite{Grenier98} (see, also \cite{AC-ARMA,CR-CMP}),
and is extended to  the Gross-Pitaevskii equation \cite{AC-GP} (see, also \cite{LZ-ARMA}) and
to the equation with (de)focusing nonlocal nonlinearities:
Schr\"odinger-Poisson  equation \cite{AC-SP,LT-MAA} (see, also \cite{LL-EJDE,ZhSIAM}) and Hartree equation \cite{CM-AA}.

Our goal is to justify \eqref{eq:h1expansion}.
The natural choice of the main part $(b_0,\phi_0)$ may be $(a^h,\phi^h)_{|h=0}$.
Letting $h=0$ in \eqref{eq:sys}, we obtain a hydrodynamical system
\begin{equation}\label{eq:lsys}
	\left\{
	\begin{aligned}
		&\partial_\tau b_0 + \nabla \phi_0 \cdot \nabla b_0 + \frac12 b_0 \Delta \phi_0 = 0,\\
		&\partial_\tau \phi_0 + \frac12 |\nabla \phi_0|^2 + \l \tau^{\g-2}(|y|^{-\g}*|b_0|^2) =0, \\
		& b_{0|\tau=0}=a_0, \quad \phi_{0|\tau=0}=0.
	\end{aligned}
	\right.
\end{equation}
For a Sobolev data $a_0$,
\eqref{eq:lsys} has a unique local solution (Theorem \ref{thm:main}).
Then, the main task is to determine $h^1$-term $(b_1,\phi_1)$.
The difficulty of finding $h^1$-term lies in the following two respects:
\begin{enumerate}
\item The equaiton \eqref{eq:sys} itself depends on $h$ through
the term $i\frac{h}{2}\Delta a^h$.
\item The initial time of \eqref{eq:sys} tends to $\tau=0$ at a speed
$h^{\frac{\al}{\g-\al}}$.
\end{enumerate}
We will see that the first becomes crucial when we consider the asymptotic behavior of the solution on the final layer,
and that a suitable choice of $h^1$-term is the key for overcoming this difficulty.
We use $(b_{\mathrm{equ}},\phi_{\mathrm{equ}})$ defined by \eqref{eq:tbtp1}, below, as an $h^1$-term 
and show an asymptotic behavior of $\psi^h$ 
for $\tau\in [h^{\frac{\al}{\g-\al}},T]$ when $\al \ge1$, where $T$ is independent of $h$
 (Theorem \ref{thm:main}).
 By \eqref{eq:timetranslation},
 $\tau\in [h^{\frac{\al}{\g-\al}},T]$ is equivalent to $t\in[0,1-T^{-1}\eps^{\al/\g}]$,
from the initial time to the final layer.
On the other hand,
the second becomes crucial in the supercritical case $\alpha<1$.
This is because the moving speed of initial time becomes too slow.
In this case, we need three more kinds of correction terms whose order are 
between $h^0$ and $h^1$. 
With them, we describe the asymptotic behavior also for $\al \in]0,1[$.
A heuristic observation on these correction terms
is in Section \ref{subsec:summary2}.
The rigorous result is in Theorem \ref{thm:sscase}.

The well-prepared data \eqref{eq:wpdata}
is closely related to  the second problem listed above.
The function $(b_0,\phi_0)$ in \eqref{eq:wpdata}
is the solution of \eqref{eq:lsys}.
If we employ the well-prepared data and consider \eqref{eq:r4} with \eqref{eq:wpdata},
then \eqref{eq:sys} changes into
\begin{equation}\label{eq:wpsys}
	\left\{
	\begin{aligned}
		&\partial_\tau a^h + \nabla \phi^h \cdot \nabla a^h + \frac12 a^h \Delta \phi^h = i\frac{h}{2}\Delta a^h,\\
		&\partial_\tau \phi^h + \frac12 |\nabla \phi^h|^2 + \l \tau^{\g-2}(|y|^{-\g}*|a^h|^2) =0, \\
		& a^h_{|\tau=h^{\frac{\al}{\g-\al}}}=b_0(h^{\frac{\al}{\g-\al}}), \quad \phi^h_{|\tau=h^{\frac{\al}{\g-\al}}}=\phi_0(h^{\frac{\al}{\g-\al}}).
	\end{aligned}
	\right.
\end{equation}
The initial time is still moving, however,
the only difference between \eqref{eq:wpsys} and \eqref{eq:lsys}
is the existence of $i\frac{h}2\Delta a^h$.
Thus, we will see that we do not meet with the second problem any longer.
This point is discussed in Sections \ref{subsec:wpdata}. 

We also consider the problem of global existence
of the solution to \eqref{eq:lsys}.
With the notation $(\rho,v) := (|b_0|^2,\nabla \phi_0)$,
\eqref{eq:lsys} is the compressible Euler equation with 
time-dependent pressure term of Hartree type:
\begin{equation}\label{eq:cE}
	\left\{
	\begin{aligned}
		&\partial_\tau \rho + \mathrm{div} (\rho v) = 0,\\
		&\partial_\tau v + (v \cdot \nabla)v + \l \tau^{\g-2}\nabla (|x|^{-\g}*\rho) =0,\\
		&\rho_{|\tau=0}=|a_0|^2, \quad v_{|\tau=0}=0. 
	\end{aligned}
	\right.
\end{equation}
We adapt results in \cite{ELT-IUMJ,MaRIMS} (see  also \cite{MP-JJAM,PeJJAM})
to prove that $C^1$-solution to \eqref{eq:cE} cannot be global in time,
in several situations
(Theorems \ref{thm:main2} and \ref{thm:main3}).
This implies \eqref{eq:h1expansion}
breaks down before caustic $t=1$ (see Section \ref{subsec:summary}).

\subsection{Main result I}
To state our result precisely, we introduce some notation.
For $n \ge 3$, $s>n/2+1$, $p\in [1, \I]$, and $q \in [1, \I]$,
we define a function space $Y^s_{p,q}(\R^n)$ by
\begin{equation}\label{def:Y}
	Y^s_{p,q}(\R^n) = \overline{C_0^\I (\R^n)}^{\norm{\cdot}_{Y^s_{p,q}(\R^n)}}
\end{equation}
with norm
\begin{equation}\label{def:Y2}
	\norm{\cdot}_{Y^s_{p,q}(\R^n)} :=
	\norm{\cdot}_{L^p(\R^n)} + \norm{\nabla \cdot}_{L^q(\R^n)}
	+ \norm{\nabla^2 \cdot}_{H^{s-2}(\R^n)}.
\end{equation}
We denote $Y^s_{p,q}=Y^s_{p,q}(\R^n)$, for short.
For $q<n$, we use the notation $q^* = nq/(n-q)$.
This space $Y^s_{p,q}$ is a modification of the Zhidkov space $X^s$,
which is defined, for $s>n/2$, by $X^s(\R^n) := \{  f \in L^\I(\R^n) | \nabla f \in H^{s-1}(\R^n) \}$.
The Zhidkov space was introduced in \cite{Zhidkov} (see, also \cite{Gallo}).
Roughly speaking, the exponents $p$ and $q$ in $Y^s_{p,q}$ indicate the decay rates at
the spatial infinity of a function and of its first derivative, respectively.
Moreover, the Zhidkov space $X^s$ corresponds to $Y^s_{\I,2}$ in a sense, if $n \ge 3$. 
We discuss these points more precisely in Section \ref{subsec:spaceY}.
We also note that $Y^s_{2,2}$ is the usual Sobolev space $H^s$.
We use the following notation: 
$Y^\I_{p,q} := \cap_{s>0}Y^s_{p,q}$;
for intervals $I_1$ and $I_2$ of $[1,\I]$,
$Y^s_{I_1,q} := \cap_{p\in I_1} Y^s_{p,q}$
and $Y^s_{p,I_2} := \cap_{q\in I_2} Y^s_{p,q}$.
These notation are sometimes used simultaneously, for example
$Y^\I_{I_1,I_2} := \cap_{s>0,p\in I_1,q\in I_2} Y^s_{p,q}$.
We also use the operator
\[
	|J^\eps|^s = e^{i\frac{|x|^2}{2\eps(t-1)}} |(1-t)\nabla|^s e^{-i\frac{|x|^2}{2\eps(t-1)}}.
\]
This is the scaled version of the Galilean operator
and is suitable for a study of rapid phases other than $e^{i|x|^2/2\eps(t-1)}$.

We also introduce the following systems for a pair
$(b_{\mathrm{equ}},\phi_{\mathrm{equ}}) $:
\begin{equation}\label{eq:tbtp1}
	\(
	\begin{aligned}
		&\d_\tau b_{\mathrm{equ}} + \nabla \phi_{\mathrm{equ}} 
		\cdot \nabla b_0
		+ \nabla \phi_0 \cdot \nabla b_{\mathrm{equ}}
		+  \frac{1}{2} b_{\mathrm{equ}} \Delta \phi_0 + \frac{1}{2} b_0 \Delta \phi_{\mathrm{equ}}
		= \frac{i}{2} \Delta b_0, \\
		&\d_\tau \phi_{\mathrm{equ}} + \nabla \phi_0 \cdot  \nabla \phi_{\mathrm{equ}}
		+ \l \tau^{\g-2} (|y|^{-\g}* 2 \Re \overline{b_0} b_{\mathrm{equ}} ) = 0,
	\end{aligned}
	\right.
\end{equation}
where $(b_0,\phi_0)$ is a solution of \eqref{eq:lsys}.
This will be posed with the zero data
\begin{align}\label{eq:tbtp2}
	b_{\mathrm{equ}|\tau=0}={}&0, &
	\phi_{\mathrm{equ}|\tau=0}={}& 0
\end{align}
or the data
\begin{align}\label{eq:tbtp3}
	b_{\mathrm{equ}|\tau=0}={}&0, &
	\phi_{\mathrm{equ}|\tau=0}={}& \l(|y|^{-\g}*|a_0|^2)/(\g-1).
\end{align}

\begin{notation}\label{not:asymp}
Let $T>0$ and $X$ be a Banach space. 
Let  $\{k_j\}$ be an increasing sequence of real number,
$\phi(t,x) \in C([0,T];X)$ be a function ,
and $\{\phi_j\}$ be a sequence of function in $X$.
We write
\[
	\phi(t,x) \asymp \sum_{j=1}^\I t^{k_j} \phi_j \IN X
\]
if it holds that
\[
	\norm{\phi(t,x) - \sum_{j=1}^J t^{k_j} \phi_j}_{X} = o(t^{k_J})
\]
as $t\to0$ for all $J\ge1$.
\end{notation}
We now state our main result.
To avoid complicity, here we state the result for $\al \ge 1$;
the linear and nonlinear WKB cases.
For the supercritical WKB case $\al<1$, see Theorem \ref{thm:sscase}.
\begin{assumption}\label{asmp:1}
Let $n \geq 4$ and $\l \in \R$.
The constants $\g$ and $\al$ satisfy $\max(1,n/2-2) < \g \leq n-2$
and $0< \al < \g$, respectively.
The initial data $a_0 \in H^\I$.
\end{assumption}
\begin{theorem}\label{thm:main}
Let assumption \ref{asmp:1} be satisfied. Assume $\al \ge 1$.
Then, there exists an existence time $T>0$ independent of $\eps$.
There also exist
$(b_0,\phi_0)$, $(b_{\mathrm{equ}}, \phi_{\mathrm{equ}}) \in C([0,T];H^\I \times Y^\I_{(n/\g,\I],(n/(\g+1),\I]})$
such that:
\begin{enumerate}
\item $\phi_0(\tau,y) \asymp \sum_{j=1}^\I \tau^{\g j -1} \varphi_j(y)$ 
in $Y^\I_{(n/\g,\I],(n/(\g+1),\I]}$.
\item The solution $u^\eps$ to \eqref{eq:r4} with \eqref{eq:odata}
satisfies the following asymptotics for all $s \ge 0$:
\begin{equation}\label{eq:asymptotics}
	\sup_{t \in [0,1-T^{-1}\eps^{\al/\g}]} \Lebn{|J^\eps|^s \( u^\eps(t) e^{-i\Phi^\eps (t)}
	- \frac{1}{(1-t)^{n/2}} A^\eps(t) e^{i\frac{|\cdot|^2}{2\eps(t-1)}} \)}{2} \to 0
\end{equation}
as $\eps \to 0$ with
\begin{equation}\label{def:Phi}
	\Phi^\eps(t,x) = \eps^{\frac{\al}{\g}-1} 
	\phi_0\(\frac{\eps^{\frac{\al}{\g}}}{1-t},\frac{x}{1-t}\) 
\end{equation}
and
\begin{equation}\label{def:A}
	A^\eps(t,x) = 
	b_0\(\frac{\eps^{\frac{\al}{\g}}}{1-t},\frac{x}{1-t} \)
	\exp \(i\phi_{\mathrm{equ}}\(\frac{\eps^{\frac{\al}{\g}}}{1-t},\frac{x}{1-t} \)\) ,
\end{equation}
where $(b_0,w_0)$ solves \eqref{eq:lsys} and $(b_{\mathrm{equ}},\phi_{\mathrm{equ}})$ solves
\eqref{eq:tbtp1} with \eqref{eq:tbtp2} if $\al>1$ and
 with \eqref{eq:tbtp3} if $\al=1$.
\end{enumerate}
\end{theorem}
\begin{remark}\label{rmk:main}
\begin{enumerate}
\item By the definition of ``$\asymp$'' sign,
the expansion $\phi_0(\tau,y) \asymp \sum_{j=1}^\I\tau^{\g j -1} \varphi_j(y)$
 implies
\begin{equation}\label{eq:p-exp}
 	\phi_0(\tau) = \sum_{j=1}^J \tau^{\g j -1} \varphi_j + o(\tau^{\g J-1})
 	\IN Y^\I_{(n/\g,\I],(n/(\g+1),\I]}
\end{equation}
as $\tau\to 0$ for all $J \ge 1$.
\item In Theorem \ref{thm:main},
we only need $a_0 \in H^{s_0}$ with some $s_0 > n/2 + 3$.
Then, $(b_0,\phi_0)$ belongs to $C([0,T],H^{s_0} \times Y^{s_0+2}_{(n/\g,\I],(n/(\g+1),\I]})$ and 
$(b_{\mathrm{equ}},\phi_{\mathrm{equ}})$ belongs to $C([0,T],H^{s_0-2} \times Y^{s_0}_{(n/\g,\I],(n/(\g+1),\I]})$.
The asymptotics \eqref{eq:asymptotics} holds for any $s \in [0,s_0-4]$.
\item All $\varphi_j$ are given explicitly (but inductively) in terms of $a_0$.
For example, $\varphi_1=\l(|x|^{-\g}*|a_0|^2)/(1-\g)$ and
\begin{align*}
	\varphi_2={}&-\frac{\l^2}{2(\g-1)^2(2\g-1)}
	|\nabla (|x|^{-\g}*|a_0|^2)|^2 \\
	&{} - \frac{\l^2}{\g(\g-1)(2\g-1)}\(|x|^{-\g} *
	(\nabla\cdot (  |a_0|^2\nabla(|x|^{-\g}*|a_0|^2))\) 
\end{align*} 
(see, Proposition \ref{prop:t-expansion}).
\item Even if the system \eqref{eq:tbtp1} is posed with zero initial condition,
its solution $(b_{\mathrm{equ}},\phi_{\mathrm{equ}})$ is not identically zero because of the presence of the (nontrivial) external force $\frac{i}{2}\Delta b_0$.
\item The choice of the initial data \eqref{eq:tbtp3} 
is the key for the analysis in the case $\al =1$.
We discuss this point more precisely in Section \ref{subsec:summary2}.
\end{enumerate}
\end{remark}
The asymptotics \eqref{eq:asymptotics} reads
\begin{equation}\label{eq:asympL2}
	u^\eps(t,x) \sim \frac{1}{(1-t)^{n/2}} A^\eps(t,x)e^{i\Phi^\eps (t,x)}
	e^{i\frac{|x|^2}{2\eps(t-1)}}
\end{equation}
as $\eps \to 0$. Indeed, this holds in $L^\I([0,1-T^{-1}\eps^{\al/\g}];L^2)$.
This explains the cascade of phase shifts.
We consider the case $1<\al(<\g)$.
Combining \eqref{eq:p-exp} and \eqref{def:Phi}, we have
\[
	\Phi^\eps(t,x) = \sum_{j=1}^J \frac{\eps^{\al j-1}}{(1-t)^{\g j-1}}
	\varphi_j \(\frac{x}{1-t}\) + o \( \frac{\eps^{\al J-1}}{(1-t)^{\g J-1}}\).
\]
We set $g^\eps_j(t,x) = \frac{\eps^{j\al -1}}{(1-t)^{j\g -1}} \varphi_j (\frac{x}{1-t})$.
Then, the above asymptotics yields
\begin{align*}
	\Phi^\eps(t,x) &\sim 0 &  \text{for }&1-t \gg \eps^{\frac{\al-1}{\g-1}} ,\\
	\Phi^\eps(t,x) &\sim g^\eps_1(t,x) &  \text{for }&1-t \gg \eps^{\frac{2\al-1}{2\g-1}} , \\
	\Phi^\eps(t,x) &\sim g^\eps_1(t,x) + g^\eps_2(t,x) &  \text{for }&1-t \gg \eps^{\frac{3\al-1}{3\g-1}} , \\
	&\vdots & &\vdots \\
		\Phi^\eps(t,x) &\sim \sum_{j=1}^J g^\eps_j(t,x) &  \text{for }&1-t \gg \eps^{\frac{J\al-1}{J\g-1}} , \\
	&\vdots & &\vdots
\end{align*}
as $\eps\to0$.
On the other hand, the amplitude $A^\eps$ satisfies
\[
	A^\eps(t,x) = a_0\(\frac{x}{1-t}\) + o(1)
\]
as $\eps^{\frac{\al}{\g}}/(1-t) \to 0$.
Substitute these expansions to \eqref{eq:asympL2} to obtain
\begin{align*}
	u^\eps &\sim  
	v^\eps_{\mathrm{lin}}=
	\frac{1}{(1-t)^{n/2}} a_0\(\frac{x}{1-t}\) e^{i\frac{|x|^2}{2\eps(t-1)}}
	&  \text{for }&1-t \gg \eps^{\frac{\al-1}{\g-1}} ,\\
	u^\eps &\sim
	v^\eps_{\mathrm{lin}}e^{ig^\eps_1(t,x)}
	&  \text{for }&1-t \gg \eps^{\frac{2\al-1}{2\g-1}} , \\
	u^\eps &\sim 
	v^\eps_{\mathrm{lin}}e^{ig^\eps_1(t,x)+ig^\eps_2(t,x)}
	&  \text{for }&1-t \gg \eps^{\frac{3\al-1}{3\g-1}} , \\
	&\vdots & &\vdots \\
	u^\eps &\sim
	v^\eps_{\mathrm{lin}}e^{i\sum_{j=1}^J g^\eps_j(t,x)}
	&  \text{for }&1-t \gg \eps^{\frac{J\al-1}{J\g-1}} , \\
	&\vdots & &\vdots
\end{align*}
Recall that $v^\eps_{\mathrm{lin}}$, given in \eqref{eq:linear},
is the approximate solution for the linear solution $w^\eps$.
One sees that the solution behaves like a free solution in the region
$1-t\gg \eps^{\frac{\al-1}{\g-1}}$ where the initial time $t=0$ lies,
and that, at each boundary layer of
size $1-t \sim \eps^{\frac{J\al-1}{J\g-1}}$
(the \textit{$J$-th boundary layer}),
a new phase associated with $g^\eps_J$ becomes relevant.

\subsection{Main result II}
Our next result is the non-existence of a global solution to \eqref{eq:lsys}.
We further assume the radial symmetry and $\g=n-2$ ($n\ge 3$)
in \eqref{eq:cE}.
A suitable change of $\l$
yields the radial compressible Euler-Poisson equations
\begin{equation}\label{eq:EPr}
\left\{
\begin{aligned}
&\partial_\tau (\rho r^{n-1}) + \partial_r (\rho v r^{n-1}) = 0, \\
&\partial_\tau v + v \partial_r v -\l \tau^{n-4} \partial_r V_{\mathrm{p}} = 0, \\
& \partial_r (r^{n-1}\partial_r V_{\mathrm{p}}) = \rho r^{n-1},\\
&\rho_{|\tau=0}=|a_0|^2, \quad v_{|\tau=0}=0,
\end{aligned}
\right.
\end{equation}
where $r:=|x|$.
We define the ``mean mass'' $M_0$ in $\{|x|\le r\}$ by 
\[
	M_0(r) := \frac{1}{r^n}\int_0^r |a_0(s)|^2 s^{n-1} ds.
\]
\begin{theorem}\label{thm:main2}
Let $\l < 0$ and $n \ge 4$.
For every nonzero initial amplitude $a_0 \in C^1$, the solution to \eqref{eq:EPr} breaks down no latter than
\[
	T^* = \(\frac{(n-2)(n-3)}{|\l| \sup_{r \ge 0} M_0(r)}\)^{\frac1{n-2}}<\I.
\]
\end{theorem}
\begin{theorem}[\cite{ELT-IUMJ}, Theorem 5.10]\label{thm:main3}
Let $\l>0$ and $n=4$.
The radial $C^1$-solution $(a,v)$ to \eqref{eq:EPr} is global if and only if
the initial amplitude $a_0\in C^1$ satisfies 
\[
	|a_0(r)|^2 \ge 2 M_0(r) = \frac{2}{r^4}\int_0^r |a_0(s)|^2 s^3 ds
\]
for all $r \ge 0$. In particular, if $a_0 \in L^2(\R^4)$ then the solution breaks down in finite time. 
Moreover, the critical time is given by
\[
	\tau_c = \(\frac{2}{\l \max_{r  > 0} \( 2M_0(r) - |a_0(r)|^2\)}\)^{\frac12}.
\]
\end{theorem}

\begin{example}[An example of finite-time breakdown]\label{ex:blowup}
Consider the equation \eqref{eq:EPr}. 
Let $n=4$ and $\l>0$.
Suppose 
\[
	a_0(x) = a_0(r) = r^{-\frac{5}{2}} e^{-\frac{1}{2r}}.
\]
Then, an elementary calculation shows
\[
	M_0(r) = r^{-4} e^{-\frac{1}{r}}
\]
and so that $2M_0(r)-|a_0|^2 = (2r-1)r^{-5} e^{-\frac{1}{r}}$
takes its maximum at $r_0=(7+\sqrt{17})/16=0.6951 \dots$
which is the root of
\[
	8r_0^2 -7r_0 +1 = 0.
\]
By Theorem \ref{thm:main3}, the critical time is
\[
	\tau_c = \(\frac{2}{|\l| \( 2M_0(r_0) - |a_0(r_0)|^2\)}\)^{\frac12} =
\sqrt{\frac{3405+827\sqrt{17}}{\l 2^{13}}} e^{\frac{3(7-\sqrt{17})}{32}}.
\]
In fact, the solution to \eqref{eq:EPr} is given by
\begin{align*}
	|a|^2(\tau,X(t,R)) ={}& \frac{2 R^2}
	{X^2(\tau,R) (2R^5e^{\frac{1}{R}} - \l(2R-1) \tau^2)
	}, \\
	v(\tau,X(\tau,R)) ={}& \frac{\l \tau}{X(\tau,R) R^2 e^{\frac{1}{R}}}, 
\end{align*}
where
\[
	X(\tau,R) = R\sqrt{1+\l  R^{-4}e^{-\frac{1}{R}} \tau^2}.
\]
At the time $\tau=\tau_c$, the characteristic curves ``touch'' at $r=r_c:=X(t_c,r_0)$,
that is, we have $(\partial_R X)(\tau_c,r_0)=0$, which is one of 
the sufficient and necessary
condition for finite-time breakdown (see, \cite{ELT-IUMJ}).
More explicitly, we can see that, as $\tau$ tends to $\tau_c$, the amplitude $|a|^2$ blows up at
$r_c$ since the denominator
\[
	2r_0^5e^{\frac{1}{r_0}} - \l(2r_0-1) t^2
\]
tends to zero as $\tau \to \tau_c$.
We illustrate the calculation in Remark \ref{rem:indicator}.

In this example, $a_0 \in H^\I \subset C^1$ and so
Theorem \ref{thm:main} holds.
We see that, however, the asymptotics \eqref{eq:asymptotics} is valid only for $\eps^{\frac{\al}{\g}}/(1-t)<\tau_c$ and
it cannot hold for $\eps^{\frac{\al}{\g}}/(1-t) \ge \tau_c$.
\end{example}

The rest of the paper is organized as follows.
In Section \ref{sec:knownresults},
we make a summary of the results in this paper with previous results.
Section \ref{sec:pre} is devoted to preliminary results.
We prove Theorem \ref{thm:main} in Section \ref{sec:proof}.
The strategy of the proof is illustrated rather precisely in Section \ref{subsec:strategy}.
In Section \ref{sec:sscase}, we treat the supercritical WKB case $\al <1$ 
(Theorem \ref{thm:sscase}).
The well-prepared data is discussed in Section \ref{subsec:wpdata}
Finally, we prove Theorems \ref{thm:main2} and \ref{thm:main3} in
Section \ref{sec:blowup}.

\section{Summary}\label{sec:knownresults}
\subsection{Cascade of phase shifts in the linear WKB case}\label{subsec:summary}
We first discuss about the linear WKB case $\al>1$.
According to Theorems \ref{thm:main}, \ref{thm:main2}, and \ref{thm:main3}, 
we summarize the result in this case as follows.
Recall several boundary layers:
\begin{equation}\label{eq:layerlist}
\begin{aligned}
\text{initial time:}& & t&{}=0 
&&\longleftrightarrow  1-t= 1
,\\
\text{first layer:}& & t&{}=1-\Lambda_1\eps^{\frac{\al-1}{\g-1}}
&& \longleftrightarrow  1-t = \Lambda_1\eps^{\frac{\al-1}{\g-1}}
,\\
J\text{-th layer:}& & t&{}=1-\Lambda_J\eps^{\frac{J\al-1}{J\g-1}}
&& \longleftrightarrow 1-t = \Lambda_J\eps^{\frac{J\al-1}{J\g-1}}
,\\
\text{final layer:}& & t&{}=1-T^{-1}\eps^{\frac{\al}{\g}}
&& \longleftrightarrow  1-t = T^{-1}\eps^{\frac{\al}{\g}}
,
\end{aligned}
\end{equation}
where $\Lambda_J$ and $T$ are positive constants.
\smallbreak

$\bullet$ {\bf From the initial time to the first layer}.
Before the first layer, that is, for $1-t \gg \eps^{\frac{\al-1}{\g-1}}$,
the behavior of the solution is the same as in the linear case at leading order.
Indeed, the asymptotics \eqref{eq:asymptotics} implies that the solution behaves like 
$v^\eps_{\text{lin}}$ defined in \eqref{eq:linear} 
for $1-t \gg \eps^{\frac{\al-1}{\g-1}}$.
The phase shifts disappear: From \eqref{def:Phi} and the time expansion 
$\phi_0(\tau) \asymp \sum_{j=1}^\I \tau^{\g j -1} \varphi_j$, we see
\[
	\norm{\Phi^{\eps}(t)}_{L^\I(\R^n)}
	\le C \eps^{\frac{\al}{\g}-1}
	\sum_{j=1}^\I  \(\frac{\eps^{\frac{\al}{\g}}}{1-t}\)^{\g j-1} 
	\norm{\varphi_j}_{L^\I(\R^n)} \ll 1
\]
if $1-t \gg \eps^{\frac{\al-1}{\g-1}}$.
Moreover, the amplitude tends to a rescaling of $a_0$:
By \eqref{def:A}, 
\[
	A^\eps(t,x) \to  b_0\(0,\frac{x}{1-t} \)
	\exp \(i\phi_{\mathrm{equ}}\(0,\frac{x}{1-t} \)\) = a_0\( \frac{x}{1-t}\)
\]
since $\eps^{\al/\g}/(1-t) \ll 1$ 
for $1-t \gg \eps^{\frac{\al-1}{\g-1}}$.
This agrees with the analysis in \cite{MaAHP}.
On the other hand, on the first layer $1-t = \Lambda_1 \eps^{\frac{\al-1}{\g-1}}$
the nonlinear effect becomes relevant.
The term ${\tau^{\g-1} \varphi_1}$ in the sum\footnote{ 
This sum is the formal one:
Here, we say ``sum'' in the sense that
$\phi_0(\tau) \asymp \sum_{j=1}^\I \tau^{\g j -1} \varphi_j$.
}
$\sum_{j=1}^\I \tau^{\g j -1} \varphi_j$ is no longer negligible.
\smallbreak

$\bullet$ {\bf From the first layer to the $J$-th layer}.
Soon after the first layer, the solution becomes strongly oscillatory by
${\tau^{\g-1}\varphi_1}$. 
Between the first and the second layers,
only ${\tau^{\g-1}\varphi_1}$ is effective because
\begin{multline*}
	\norm{\Phi^{\eps}(t)-
	\eps^{\frac{\al}{\g} -1 }\(\frac{\eps^{\frac{\al}{\g}}}{1-t}\)^{\g-1}
	\varphi_1\(\frac{\cdot}{1-t} \) }_{L^\I(\R^n)} \\
	\le \eps^{\frac{\al}{\g}-1}
	\sum_{j=2}^\I  \(\frac{\eps^{\frac{\al}{\g}}}{1-t}\)^{\g j-1} 
	\norm{\varphi_j}_{L^\I(\R^n)} \ll 1
\end{multline*} 
holds if $1-t \gg \eps^{\frac{2\al-1}{2\g-1}}$.
When we reached to the second layer $1-t = \Lambda_2 \eps^{\frac{2\al-1}{2\g-1}}$, the phase $\tau^{2\g-1}\varphi_2$ become relevant.
Similarly, at each $J$-th layer $1-t = \Lambda_J \eps^{\frac{J\al-1}{J\g-1}}$ new phase $\varphi_J$ becomes relevant.
This is the cascade of phase shift.
In this regime, $\eps^{\frac{\al}{\g}}/(1-t)$ converges to zero as $\eps \to 0$, 
and so that the amplitude still stays the linear one.
\smallbreak

$\bullet$ {\bf On the final layer}.
After a countable number of boundary layers, we reach to the final layer.
In this layer, $\eps^{\frac{\al}{\g}}/(1-t)$ does not tend to zero any longer.
Therefore, the asymptotics of amplitude changes into a nonlinear one.
It turns out that the ratio $T:=\eps^{\al/\g}/(1-t)$ plays the crucial role.
If $T \ge 0$ is small, then the asymptotic behavior of the solution is described by
the asymptotics \eqref{eq:asymptotics}.
It is essential to use $\Phi^\eps$ and $A^\eps$ defined
in \eqref{def:Phi} and \eqref{def:A}, respectively.
Thanks to the nontrivial remainder term given by $\phi_{\mathrm{equ}}$, 
\eqref{eq:asymptotics} gives the asymptotic behavior of the solution on the final layer.

On the other hand, \eqref{eq:asymptotics} breaks down at $T=T^*<\I$ 
in several cases (Theorems \ref{thm:main2} and \ref{thm:main3}).
It is because the nonlinear effects cause a formation of singularity.
These focal points are moving:
If the focal point is 
\[
	\(\frac{\eps^{\frac{\al}{\g}}}{1-t_c}, \frac{x_c}{1-t_c}\) = ( T^*,X^*),
\]
then, in the $(t,x)$-coordinates, this focal point is
\begin{align*}
	t_c =&{} 1-(T^*)^{-1} \eps^{\frac{\al}{\g}}, &
	x_c =&{} X^*(T^*)^{-1} \eps^{\frac{\al}{\g}},
\end{align*}
which tends to $(t,x)=(1,0)$ as $\eps\to0$.
Example \ref{ex:blowup} gives an example of this type blowup.
\smallbreak

$\bullet$ {\bf After the final layer}.
It remains open what happens after the breakdown of \eqref{eq:asymptotics}.
However, we can at least expect that more rapid oscillations might not appear after the final layer.
This is because, in the case of $\l>0$, $2\le \g< 4$, and $\al\in]\max(\g/2,1),\g[$,
the order of the upper bound of $\Lebn{J^\eps u^\eps}{2}$ stays $\eps^{1-\al/\g}$ even after the caustic (\cite{MaAHP}).
Recall that  $J^\eps = e^{i|x|^2/2\eps(t-1)}i(t-1)\nabla e^{-i|x|^2/2\eps(t-1)}$
filters out the main quadratic phase.
Therefore, this divergent upper bound implies that
the order of magnitude of energy of the oscillation
\emph{other than the main quadratic phase}
is at most $\eps^{1-\al/\g}$.

\subsection{Cascade of phase shifts in the nonlinear and supercritical WKB case}
\label{subsec:summary2}
We now turn to the case $\al < 1$, the supercritical WKB case.
Our goal is to obtain a WKB-type approximation similar to \eqref{eq:asymptotics},
which explains the cascade of phase shifts and gives the asymptotic profile
of the solution before and on the final layer.
We \emph{do not} use modified the initial data \eqref{eq:wpdata}
and keep working with \eqref{eq:r4}--\eqref{eq:odata}.
As stated above, the analysis of 
 \eqref{eq:r4}--\eqref{eq:odata}
is reduced to the analysis of \eqref{eq:sys} via the transform \eqref{eq:spct}.
Recall that $h=\eps^{1-\al/\g}$.
The difficulty is the following:
\begin{enumerate}
\item The equaiton \eqref{eq:sys} itself depends on $h$ through
the term $i\frac{h}{2}\Delta a^h$.
\item The initial time of \eqref{eq:sys} tends to $\tau=0$ at a speed
$h^{\frac{\al}{\g-\al}}$.
\end{enumerate}
The second is the main point
of the supercritical case $\alpha<1$.
In this section, we discuss by heuristic arguments
 what happens,
what is the problem, and how we can overcome it.
Then, it turns out that the situation becomes more complicated,
and the cascade of phase shifts phenomena involves more phase shifts
and boundary layers than the linear WKB case.
The rigorous result is in Theorem \ref{thm:sscase}
(see, also Remark \ref{rmk:type}).
\smallbreak

$\bullet$ {\bf The transpose of the initial time and the first boundary layer}.
We see from \eqref{eq:layerlist} that the relation $\al<1$
causes the transpose
of the initial time $t=0$ and the first boundary layer, that is,
the initial time $t=0$ lies beyond the first boundary layer $t =1-\Lambda_1 \eps^{\frac{\al-1}{\g-1}}$ because $\Lambda_1 \eps^{\frac{\al-1}{\g-1}} \gg 1$ for small $\eps$.
It means that there is no linear regime and the behavior of the solution involves
nonlinear effects at leading order soon after the initial time.
\smallbreak

$\bullet$ {\bf The nonlinear behavior of the solution at the initial time}.
As a matter of fact, the behavior of $u^\eps$ is already ``nonlinear'' at $t=0$.
More precisely, the principal part of the phase shift of the solution $u^\eps$
other than $\exp (i |x|^2/2\eps(t-1))$
(we call this as a \emph{principal nonlinear phase} of $u^\eps$) for $\al < 1$ is given by
\begin{equation}\label{eq:Pphase}
	\exp \( i \eps^{\frac{\al}{\g}-1} \phi_0\(\frac{\eps^{\frac{\al}{\g}}}{1-t},\frac{x}{1-t}\) \),
\end{equation}
where $\phi_0$ is the solution of \eqref{eq:lsys}.
Here we remark that the principal nonlinear phase is the same as in the case $\al>1$
(In Theorem \ref{thm:main}, we denote this by $\exp (i\Phi^\eps)$).
Since $\phi_0$ is given as a solution of \eqref{eq:lsys},
the shape of $\phi_0$ is completely independent of $\al$.
Moreover, the choice of $\phi_0$ is natural because the function $\phi_0$
is the unique limit of the phase function $\phi^h$
which solves \eqref{eq:sys} with $a^h$.
Recall that $\psi^h=a^h\exp(i\phi^h/h)$ is an exact solution of \eqref{eq:r5}.
Using the expansion
$\phi_0(\tau) \asymp \sum_{j=1}^\I \tau^{\g j -1}\varphi_j(y)$,
we have
\[
	\eps^{\frac{\al}{\g}-1} \phi_0\(\frac{\eps^{\frac{\al}{\g}}}{1-t}\)
	\sim O\(\eps^{\frac{\al}{\g}-1} \(\frac{\eps^{\frac{\al}{\g}}}{1-t}\)^{\g-1}\)
\]
as long as $\eps^{\frac{\al}{\g}}/(1-t) \ll 1$.
If the right hand side is small, then the phase shift caused by nonlinear effects
is negligible and so
the solution behaves like the linear solution.
However, the right hand side is $O(\eps^{\al-1})$ at $t=0$, which
is not small if $\al<1$.
In this sense, the behavior of $u^\eps$ is nonlinear at $t=0$.
\smallbreak

$\bullet$ {\bf The initial condition as a constraint}.
On the other hand, we always have
$u^\eps(0,x)=a_0(x)\exp(-i\frac{|x|^2}{2\eps})$
since, as stated above,
we do not modify the initial data and keep working with \eqref{eq:odata}.
This initial condition seems to be
quite natural and the simplest one for this problem.
This is true if $\alpha>1$.
However, in  the supercritical case $\al<1$,
the meaning of this condition slightly changes
and this condition becomes a sort of constraint:
The appearing nonlinear effects must disappear at $t=0$.
To achieve this constraint, we need to employ some more nontrivial
phase shifts as correction terms in order to cancel out the nonlinear effect at $t=0$.
This  modification with correction terms is the heart of the matter.
The main difference between two initial data \eqref{eq:odata} and \eqref{eq:wpdata} is this point.
A use of \eqref{eq:wpdata} enables us to leave the behavior of the solution at the
initial time nonlinear.
Hence, we do not need any correction term.
\smallbreak

$\bullet$ {\bf A formal construction of correction terms}.
Intuitively, this modification is done as follows:
First, we replace $\phi_0(\tau,y)$ in \eqref{eq:Pphase}
by $\phi_0(\tau,y)-\sum_{j=1}^{k-1}\tau^{\g j-1} \varphi_j(y)$
for some $k\ge2$.
This yields the following modified principal nonlinear phase
\begin{equation}\label{eq:mPphase}
	\exp \( i \eps^{\frac{\al}{\g}-1} \left[ \phi_0\(\frac{\eps^{\frac{\al}{\g}}}{1-t},\frac{x}{1-t}\)-
	\sum_{j=1}^{k-1}\(\frac{\eps^{\frac{\al}{\g}}}{1-t}\)^{\g j-1} \varphi_j\(\frac{x}{1-t}\)
	\right] \).
\end{equation}
At $t=0$, it holds that
\[
	\eps^{\frac{\al}{\g}-1} \left[ \phi_0\(\eps^{\frac{\al}{\g}}\)-
	\sum_{j=1}^{k-1}\(\eps^{\frac{\al}{\g}}\)^{\g j-1} \varphi_j
	\right] = O(\eps^{\al k-1}).
\]
Therefore, if we take $k$ large enough then the modified approximate solution
which the above principal phase \eqref{eq:mPphase} gives
possesses the desired two properties:
The leading term of the principal nonlinear phase is the same as \eqref{eq:Pphase};
and the phase shifts tend to zero at the initial time $t=0$.
Of course, this simple modification is too rude and so 
valid only in the small neighborhood of $t=0$.
Hence, to obtain an approximation also outside the small neighborhood,
we replace each $\varphi_j(y)$ by a function
$-\phi_{\mathrm{pha},j}(\tau,y)$ which solves
a kind of $j$-th linearized system of \eqref{eq:sys} with
$-\phi_{\mathrm{pha},j}(0)=\varphi_j$.
This yields the principal nonlinear phase like
\begin{equation}\label{eq:mmPphase}
	\exp \( i \eps^{\frac{\al}{\g}-1} \left[ \phi_0\(\frac{\eps^{\frac{\al}{\g}}}{1-t},\frac{x}{1-t}\)+
	\sum_{j=1}^{k-1}\(\eps^{\frac{\al}{\g}}\)^{\g j-1}\phi_{\mathrm{pha},j}\(\frac{t\eps^{\frac{\al}{\g}}}{1-t},\frac{x}{1-t}\)
	\right] \).
\end{equation}
We note that ${t\eps^{\frac{\al}{\g}}}/(1-t)={\eps^{\frac{\al}{\g}}}/(1-t)-\eps^{\frac{\al}{\g}}$,
which is zero at $t=0$.
\smallbreak

$\bullet$ {\bf Three kinds of correction terms}.
In fact, the above modified principal nonlinear phase \eqref{eq:mmPphase}
is still insufficient.
We need two more kinds of correction terms which are essentially different
from $\phi_{\mathrm{pha},j}$.
Now, let us list all kinds of correction terms which we use:
\begin{enumerate}
\item \emph{Correction from phase}.
The first one is the above $\phi_{\mathrm{pha},j}$.
They satisfy $\phi_{\mathrm{pha},j}(0)=-\varphi_j$ and
remove the bad part of $\phi_0$.
\item \emph{Correction from amplitude}.
The amplitude $b_0$ which pairs $\phi_0$ via \eqref{eq:lsys} has the 
expansion $b_0(\tau) \asymp a_0+ \sum_{j=1}^\I \tau^{\g j} a_j$,
where $a_j$ is a space function defined by $a_0$
(see, Proposition \ref{prop:t-expansion}).
Hence, the principal part of the amplitude of $u^\eps$
(\emph{principal amplitude} of $u^\eps$) is 
\[
	b_0\(\frac{\eps^{\frac{\al}{\g}}}{1-t},\frac{x}{1-t}\)
	\sim a_0\(\frac{x}{1-t}\) + \sum_{j=1}^\I
	\(\frac{\eps^{\frac{\al}{\g}}}{1-t}\)^{\g j} a_j\(\frac{x}{1-t}\)
\]
for $\eps^{\frac{\al}{\g}}/(1-t) \ll 1$.
In particular, at $t=0$ we have
\begin{equation}\label{eq:mamplitude}
	b_0 \( \eps^{\frac{\al}{\g}}, x \) \sim a_0(x)
	+ \sum_{j=1}^{k-1} \eps^{\al j}a_j + O(\eps^{\al k}).
\end{equation}
The principal amplitude converges to the given initial data for all $\al>0$,
and so one might expect this is harmless.
However, it is understood that, when we try to get the pointwise estimate of solution
via Grenier's method, we must take $\eps^1$-term of the initial amplitude
into account because it affects (implicitly) the approximate solution at leading order
(see, \cite{AC-ARMA,CM-AA,Grenier98}).
Therefore, we must remove $\eps^{\al j}a_j$ for all $j \ge 1$ such that $\al j \le 1$
from \eqref{eq:mamplitude}, 
otherwise the approximate solution will differ outside a small neighborhood of $t=0$.
To do this, we construct $b_{\mathrm{amp},j}$
as a solution to a kind of $j$-th linearized system of \eqref{eq:sys}
with the condition $b_{\mathrm{amp},j}(0)=-a_j$.
At that time, there appear a phase correction $\phi_{\mathrm{amp},j}$
associated with $b_{\mathrm{amp},j}$ via the system
which $b_{\mathrm{amp},j}$ solves.
\item \emph{Correction from interaction}.
The third one comes from the structure of \eqref{eq:sys}.
As stated in introduction,
the problem boils down to 
determining the asymptotic behavior of the
solution $(a^h,\phi^h)$ to \eqref{eq:sys} up to $O(h^1)$.
Suppose that the solution has two terms $ (h^{p_1}b_1, h^{p_1}\phi_1)$ and 
$ (h^{p_2}b_2, h^{p_2}\phi_2)$ in its asymptotic expansion as $h \to 0$.
Then, the quadratic terms in \eqref{eq:sys} produce nontrivial $h^{p_1+p_2}$-terms.
For example, 
$\nabla \phi^h \cdot \nabla a^h$ has the terms
$h^{p_1+p_2}\nabla \phi_1 \cdot \nabla b_2$ and 
$h^{p_1+p_2}\nabla \phi_2 \cdot \nabla b_1$ in its expansion.
Again by \eqref{eq:sys}, this implies that $(\partial_t b^h,\partial_t \phi^h)$
(and so $(b^h, \phi^h)$ itself) also contains a $h^{p_1+p_2}$-term in its expansion.
Repeating this argument, we see that $(b^h, \phi^h)$ has
$h^p$-terms for all $p$ given by $p=lp_1+mp_2$ with integers $l,m \ge0$
such that $l+m \ge 1$.
So far, we have already obtained two kinds of correction terms, $\phi_{\mathrm{pha},j}$
and $\phi_{\mathrm{amp},j}$.
Therefore, they interact each other to produce the third correction terms $\phi_{\mathrm{int},j}$.
\end{enumerate}
\smallbreak

$\bullet$ {\bf The supercritical cascade of phase shifts}.
We are now in a position to understand the cascade of phase shifts phenomena
in the supercritical case.
With the above three kinds of correction terms, we can describe the asymptotic
behavior of the solution before the final layer.
Recall that, in the linear WKB case $\al>1$, the principal nonlinear phase is 
given by only one phase function $\phi_0(\tau,y)$,
and the notion of boundary layer comes from its expansion with respect to $\tau$
around $\tau=0$.
In this case, not only $\phi_0$ but also all 
$\phi_{\mathrm{pha},i}$, $\phi_{\mathrm{amp},j}$, $\phi_{\mathrm{int},k}$
produce a countable number of similar boundary layers.
Thus, the cascade of phase shifts involves much more 
phase shifts and boundary layers than the linear WKB case.
\smallbreak

$\bullet$ {\bf One more correction at the final boundary layer}.
So far, the second difficulty listed in the beginning of this section is solved
by three kinds of correction terms $\phi_{\mathrm{pha},i}$,
$\phi_{\mathrm{amp},j}$, and $\phi_{\mathrm{int},j}$.
We can describe with them the asymptotic behavior of the solution before the final layer.
To give it also on the final layer, we need one more correction term,
$\phi_{\mathrm{equ}}$, as in the case $\alpha > 1$.
This correction term defeats the first difficulty.
This solves \eqref{eq:tbtp1}--\eqref{eq:tbtp2} and so it is independent of $\al$.
When $\al=1$, $\phi_{\mathrm{equ}}$ changes in to
a solution of \eqref{eq:tbtp1}--\eqref{eq:tbtp3} (see Theorem \ref{thm:main}).
We next address why we need a modified $\phi_{\mathrm{equ}}$.  
\smallbreak

$\bullet$ {\bf Resonance of correction terms and the nonlinear WKB case}.
Some of the correction terms associated with 
$\phi_{\mathrm{pha},i}$, $\phi_{\mathrm{amp},j}$, $\phi_{\mathrm{int},k}$,
and $\phi_{\mathrm{equ}}$ may have the same order.
This phenomena is a kind of resonance.
The nonlinear case $\al=1$ is the simplest example.
In this case, it happens that $\phi_{\mathrm{pha},1}$ and  $\phi_{\mathrm{equ}}$ have the same order.
Recall that if $\al=1$ then we work with the modified $\ti{\phi}_{\mathrm{equ}}$ solving
\eqref{eq:tbtp1}--\eqref{eq:tbtp3},
\begin{equation*}
	\(
	\begin{aligned}
		&\d_\tau \ti{b}_{\mathrm{equ}}
		+ \nabla \ti{\phi}_{\mathrm{equ}} \cdot \nabla b_0
		+ \nabla \phi_0 \cdot \nabla \ti{b}_{\mathrm{equ}}
		+  \frac{1}{2} \ti{b}_{\mathrm{equ}} \Delta \phi_0
		+ \frac{1}{2} b_0 \Delta \ti{\phi}_{\mathrm{equ}}
		= \frac{i}{2} \Delta b_0, \\
		&\d_\tau \ti{\phi}_{\mathrm{equ}}
		+ \nabla \phi_0 \cdot  \nabla \ti{\phi}_{\mathrm{equ}}
		+ \l \tau^{\g-2} (|y|^{-\g}* 2 \Re \overline{b_0} \ti{b}_{\mathrm{equ}} ) = 0, \\
		&	\ti{b}_{\mathrm{equ}|\tau=0} \equiv 0, \quad \ti{\phi}_{\mathrm{equ}|\tau=0} = \l(|y|^{-\g}*|a_0|^2)/(\g-1)=-\varphi_1
	\end{aligned}
	\right. 
\end{equation*}
(see Theorem \ref{thm:main}).
We will see later that $\phi_{\mathrm{pha},1}$ is the solution to
\begin{equation*}
	\(
	\begin{aligned}
		&\d_\tau b_{\mathrm{pha},1} + \nabla \phi_{\mathrm{pha},1} \cdot \nabla b_0
		+ \nabla \phi_0 \cdot \nabla b_{\mathrm{pha},1}
		+  \frac{1}{2} b_{\mathrm{pha},1} \Delta \phi_0
		+ \frac{1}{2} b_0 \Delta \phi_{\mathrm{pha},1} =0, \\
		&\d_\tau \phi_{\mathrm{pha},1} + \nabla \phi_0 \cdot  \nabla \phi_{\mathrm{pha},1}
		+ \l \tau^{\g-2} (|y|^{-\g}* 2 \Re \overline{b_0} b_{\mathrm{pha},1} ) = 0, \\
		&	b_{\mathrm{pha},1|\tau=0} \equiv 0, \quad \phi_{\mathrm{pha},1|\tau=0} = \l(|y|^{-\g}*|a_0|^2)/(\g-1)=-\varphi_1
	\end{aligned}
	\right. 
\end{equation*}
(see Remark \ref{rmk:type}).
Let $\phi_{\mathrm{equ}}$ be a solution to \eqref{eq:tbtp1}--\eqref{eq:tbtp2}.
One can check  from above systems that
	$\ti{\phi}_{\mathrm{equ}}=\phi_{\mathrm{equ}}+\phi_{\mathrm{pha},1}$.
Therefore, the modified correction term $\ti{\phi}_{\mathrm{equ}}$ 
is nothing but the superposition of the (usual) correction from equation $\phi_{\mathrm{equ}}$
and the correction from phase $\phi_{\mathrm{pha},1}$.

\section{Preliminary results}\label{sec:pre}
\subsection{Properties of the $Y^s_{p,q}(\R^n)$ space}\label{subsec:spaceY}
We first collect some facts about the space $Y^s_{p,q}$ defined in \eqref{def:Y}-\eqref{def:Y2}
(see, also \cite{AC-SP,CM-AA}).
\begin{enumerate}
\item $Y^{s_1}_{p,q} \subset Y^{s_2}_{p,q}$ if $s_1 \ge s_2$.
\item $Y^s_{p,q}=Y^s_{p,[\min(q,2^*),\I]}$ and so $Y^s_{p,q_1} \subset
Y^s_{p,q_2}$ if $q_1 \le q_2$.
\item If $q<n$ then $Y^s_{p,q}=Y^s_{[\min(p,q^*),\I],q}$.
It implies $Y^s_{p_1,q} \subset
Y^s_{p_2,q}$ for $p_1 \le p_2$ under $q<n$.
In particular, $Y^s_{p,q} \subset Y^s_{[2^{**},\I],[2^*,\I]}$ if $n \ge 5$,
where $2^{**} := (2^*)^*=2n/(n-4)$.
\item If $n \ge 3$ then any function $f \in X^s$ is written uniquely as $f=g + c$,
where $g \in Y^s_{\I,2}(=Y^s_{2^*,2})$ and $c$ is a constant.
\end{enumerate}
The first property is obvious by definition.
The others follow from the following lemma
which is a consequence of the Hardy-Littlewood-Sobolev inequality
and found in \cite[Th.~4.5.9]{Hormander1} or \cite[Lemma~7]{PGAHPANL}:
\begin{lemma}\label{lem:HPG}
If $\varphi \in \mathcal{D}^{'}(\R^n)$ is such that $\nabla \varphi \in L^p (\R^n)$ for $p \in ]1,n[$,
then there exists a constant $c$ such that $\varphi-c \in L^q (\R^n)$, with $1/p=1/q+1/n$.
\end{lemma}
We take a function $f \in Y^s_{p,q}$.
Then, the indices $p$ and $q$ almost indicate the decay rates at spacial infinity
of the function $f$ and its first derivative $\nabla f$, respectively.
The second property means that $\nabla f$ is always bounded and
decays at the spacial infinity so fast that $\nabla f \in L^{2^*}$.
What the third property says is that $f$ has a similar decay property.
It can be said from the fourth property that $Y^s_{\I,2} = X^s$ in a sense,
provided $n \ge 3$.
Note that every $g \in Y^s_{\I,2}$ satisfies $g \to 0$ as $|x| \to \I$ by definition \eqref{def:Y}-\eqref{def:Y2}.
On the other hand, elements of $X^s$ do not necessarily tend to zero at the spacial infinity.
We also note that if $q=2$ then we have another definition of $Y^s_{p,2}$:
\begin{equation}\label{def:Y3}
	Y^s_{p,2}(\R^n) = L^p (\R^n) \cap X^{s}(\R^n),
\end{equation}
which makes sense for $s>n/2$.
\subsection{Basic existence and approximation results}
\label{subsec:existence}
Operate $\nabla$ to the equation for $\phi^h$ in
\eqref{eq:sys} and put $w^h:=\nabla \phi^h$.
For our further application, we generalize the system slightly.
Let 
\begin{align}
	Q_1(a,v) & {} = - (v\cdot \nabla) a -\frac12 a \nabla \cdot v, \label{eq:q-1}\\
	Q_2(v_1,v_2) & {} = - (v_1 \cdot \nabla)v_2, \label{eq:q-2}\\
	Q_3(a_1,a_2) & {} = -\l \nabla(|x|^{-\g} * (a_1\overline{a_2})), \label{eq:q-3}
\end{align}
and consider a system of the following form:
\begin{equation}\label{eq:system1}
\left\{
  \begin{aligned}
	\d_t b^h ={}& c_1^h Q_1(b^h, w^h)  + Q_1(B_1^h,w^h) + Q_1(b^h,W_1^h)
	 +R_1^h + ir^h \Delta b^h, \\
	\d_t w^h ={}& c_2^h Q_2(w^h,w^h)   + Q_2(W_2^h,w^h) + Q_2(w^h,W_2^h) \\
	&{}  + f^h(t) \( c_2^h Q_3(b^h,b^h) + Q_3(B_2^h,b^h) + Q_3(b^h,B_2^h)\)
	+ R_2^h, 
  \end{aligned}
\right.
\end{equation}
\begin{align}\label{eq:system2}
	b^h_{|t=0} &= b_0^h, &
	w^h_{|t=0} &= w_0^h.
\end{align}
where $b^h$ takes complex value and $w^h$ takes real value.
Other notation will be made precise in Assumptions \ref{asmp:existence} and \ref{asmp:existence-coef}, below.
In \cite{CM-AA}, the existence of a unique solution 
for this kind of system is shown with explicit coefficients.
We shall summarize the parallel result.

\begin{assumption}[initial data]\label{asmp:existence}
Let $n \ge 3$ and $\max(n/2-2,0) < \gamma \le n-2$.
We suppose the following conditions with some $s>n/2+1$:
The initial amplitude $b_0^h \in H^s(\R^n)$ and the initial velocity
$w_0^h \in Y^{s+1}_{q_0,2}(\R^n)$ for some $q_0 \in ]n/(\gamma+1),n[$,
uniformly for $h \in [0,1]$, that is,
there exists a constant $C$ independent of $h$
such that $\tSobn{b_0^h}{s} + \tnorm{w_0}_{Y^{s+1}_{q_0,2}(\R^n)} \le C$.
\end{assumption}
\begin{assumption}[coefficients]\label{asmp:existence-coef}
Let $c_1^h$ and $c_2^h$ be complex constants bounded uniformly in $h$,
and let $r^h$ be a real constant bounded uniformly in $h$.
Suppose for some $T^*>0$ and $s>n/2+1$ that
$f^h$ is a real-valued function of time and $f^h \in L^1((0,T^*))$;
$B_i^h$ and $R_1^h$ are complex-valued functions of spacetime, and
$B_i^h \in L^1((0,T^*); H^{s+1})$ and $R_1^h\in L^1((0,T^*); H^{s})$;
$W_i^h$ and $R_2^h$ are real-valued functions of spacetime, and
$W_i^h \in L^1((0,T^*); Y^{s+2}_{\I,2} )$ and $R_2^h\in L^1((0,T^*); Y^{s+1}_{q_0,2})$.
Moreover, suppose all above functions are bounded in the corresponding norms uniformly with respect to $h$.
\end{assumption}
\begin{assumption}[existence of the limit]\label{asmp:approximate-coef}
In addition to Assumptions \ref{asmp:existence} and \ref{asmp:existence-coef},
we suppose the existence of limits of all $b_0^h$, $w_0^h$, $c_i^h$, $r^h$, $f^h$, $B_i^h$, $W_i^h$, and $R_i^h$
as $h\to 0$ in the corresponding strong topologies.
These strong limits are denoted by $b_0$, $w_0$, $c_i$, $r$, $f$, $B_i$, $W_i$, and $R_i$, respectively.
\end{assumption}
\begin{proposition}\label{prop:existence-sys}
Let Assumptions~\ref{asmp:existence} and \ref{asmp:existence-coef} be satisfied.
Then, there exists $T>0$ independent of $h$, $s$, and $q_0$,
such that for all $h \in [0,1]$ the system \eqref{eq:system1}--\eqref{eq:system2} has a
unique solution 
\begin{equation*}
  (b^h,w^h) \in C\([0,T]; H^s\times Y^{s+1}_{q_0,2}\).
\end{equation*}
Moreover, the norm of $ (b^h,w^h)$ is bounded
uniformly for $h \in [0,1]$. 
If, in addition, Assumption~\ref{asmp:approximate-coef} is satisfied, then the pair 
$(b^h,w^h)$ converges to $(b,w):=(b^h,w^h)_{|h=0}$
in $C([0,T]; H^{s-2} \times Y^{s-1}_{q_0,2})$ as $h \to 0$.
Furthermore, $(b,w)$ solves
\begin{equation*}
\left\{
  \begin{aligned}
	\d_t b ={}& c_1 Q_1(b, w)  + Q_1(B_1,w) + Q_1(b,W_1)
	 +R_1 + ir \Delta b, \\
	\d_t w ={}& c_2 Q_2(w,w)   + Q_2(W_2,w) + Q_2(w,W_2) \\
	&{}  + f(t) \( c_2 Q_3(b,b) + Q_3(B_2,b) + Q_3(b,B_2)\)
	+ R_2, \nonumber
  \end{aligned}
\right.
\end{equation*}
\begin{align*}
	b_{|t=0} &= b_0, &
	w_{|t=0} &= w_0.
\end{align*}
\end{proposition}
\begin{proof}
The key is the following energy estimate for $s>n/2+1$
\begin{align*}
	\frac{d}{dt} E^h
	\le{}& C (1+|f^h(t)|)\Big[(|c_1^h| + |c_2^h|)(E^h)^{\frac32}  \\
	&{}+(\tSobn{B^h_1}{s+1}+\tSobn{B^h_2}{s}+\tSobn{\nabla W^h_1}{s}+\tSobn{\nabla W^h_2}{s+1})E^h\Big] \\
	&{}+ C(\tSobn{R^h_1}{s} + \tSobn{\nabla R^h_2}{s})(E^h)^{\frac12},
\end{align*}
where $E^h:= \tSobn{b^h}{s}^2 + \tSobn{\nabla w^h}{s}^2$.
For more details, see the proof of Proposition 4.1 in \cite{CM-AA}.
\end{proof}
\begin{remark}
We intend to apply this proposition to the system \eqref{eq:sys}.
In that case, $f^h(t)$ corresponds to $t^{\g-2}$,
which is singular at $t=0$ if $\g<2$.
Since $f^h(t)$ is only supposed to be integrable,
we will see that the system \eqref{eq:sys}
(and so the equation \eqref{eq:r5}) has a unique solution for $\g>1$
while the singularity.
We also note that this corresponds to the fact that
the Hartree nonlinearity is short range when $\g>1$.
\end{remark}
\begin{remark}
In the convergence part of Proposition \ref{prop:existence-sys}, it can happen that $n/2+1 \ge s-1>n/2$.
In this case, we use the definition \eqref{def:Y3} instead of \eqref{def:Y}.
\end{remark}
We conclude this section with a lemma which we use for the construction of a function $\phi^h$ 
from the corresponding solution $w^h=\nabla \phi^h$ to \eqref{eq:system1}--\eqref{eq:system2}.
This lemma is a consequence of Lemma \ref{lem:HPG}.
\begin{lemma}\label{lem:int}
If $\varphi$ satisfies $|\varphi|\to 0$ as $|x|\to \I$ and $\nabla \varphi \in Y^s_{q,2}$
for some $s>n/2$ and some $q<n$, $q \le 2^*$ then $\varphi \in Y^{s+1}_{q^*,q}$.
\end{lemma}

\section{Proof of Theorem \ref{thm:main}}\label{sec:proof}
\subsection{Strategy}\label{subsec:strategy}
In this section, we illustrate the strategy of
the proof of Theorem \ref{thm:main} rather precisely.
As in Section \ref{sec:intro},
we first introduce the semiclassical conformal transform:
\begin{equation}\tag{\ref{eq:spct}}
	u^\eps(t,x) = \frac{1}{(1-t)^{n/2}} \psi^\eps \( \frac{\eps^{\frac{\al}{\g}}}{1-t}, \frac{x}{1-t} \)
	\exp\(i \frac{|x|^2}{2\eps(t-1)}\).
\end{equation}
Putting $\tau := \eps^{\frac{\al}{\g}}/(1-t)$ and $y := x/(1-t)$, we find that 
\eqref{eq:r4}--\eqref{eq:odata} becomes
\begin{align*}
	i \eps^{1-\frac{\al}{\g}} \d_\tau \psi^\eps
	+ \frac{(\eps^{1-\frac{\al}{\g}})^2}{2} \Delta_y \psi^\eps
	&= \l \tau^{\g-2} (|y|^{-\g} * |\psi^\eps|^2 )\psi^\eps, &
	\psi^\eps_{|\tau=\eps^{\al/\g} } &= a_0.
\end{align*}
Put $h=\eps^{1-\al/\g}$ and denote $\psi^\eps$ by $\psi^h$.
We note that $\eps \to 0$ is equivalent to $h\to 0$ as long as $\al < \g$.
Thus, our problem is reduced to the limit $h\to 0$ of the solution to 
\begin{align}\tag{\ref{eq:r5}}
  ih\d_\tau \psi^h +\frac{h^2}{2}\Delta \psi^h &= \l
                \tau^{\g-2} (\lvert y\rvert ^{-\g}\ast  |\psi^h|^2)\psi^h, &
	\psi^h_{|\tau=h^{\frac{\al}{\g-\al}}}(y) &= a_0(y).
\end{align}
Our strategy is to seek a solution $\psi^h$ to \eqref{eq:r5} represented as 
\begin{equation}\tag{\ref{eq:grenier}}\label{eq:psi-ap}
        \psi^h(\tau,y)=a^h(\tau,y)
        e^{i\phi^h(\tau,y)/h},
\end{equation}
with a complex-valued space-time function $a^h$ and a
real-valued space-time function $\phi^h$. Note that $a^h$
is expected to be complex-valued, even if its initial value $a_0$
is real-valued. 
Substituting the form \eqref{eq:psi-ap} into \eqref{eq:r5}, we obtain
\begin{multline*}
	-a^h \(  \d_\tau \phi^h + \dfrac{1}{2}|\nabla
	\phi^h|^2 + \l \tau^{\g-2}(|y|^{-\g}*|a^h|^2)\) \\
	+i h \( \d_\tau a^h + (\nabla \phi^h \cdot
	\nabla) a^h + \dfrac{1}{2}a^h \Delta \phi^h 
	- i \dfrac{h}{2}\Delta a^h\) = 0.
\end{multline*}
To obtain a solution of the above equation (hence, of \eqref{eq:r5}),
we choose to consider 
\begin{equation}\tag{\ref{eq:sys}}
	\left\{
	\begin{aligned}
		&\partial_\tau a^h + \nabla \phi^h \cdot \nabla a^h + \frac12 a^h \Delta \phi^h = i\frac{h}{2}\Delta a^h,\\
		&\partial_\tau \phi^h + \frac12 |\nabla \phi^h|^2 + \l \tau^{\g-2}(|y|^{-\g}*|a^h|^2) =0, \\
		& a^h_{|\tau=h^{\frac{\al}{\g-\al}}}=a_0, \quad \phi^h_{|\tau=h^{\frac{\al}{\g-\al}}}=0.
	\end{aligned}
	\right.
\end{equation}
The point is that this system can be regarded as a symmetric hyperbolic system
with semilinear perturbation.
In Section \ref{subsec:pexistence}, we first prove
that it admits a unique solution with suitable regularity (see
Proposition~\ref{prop:p-existence}), hence providing a solution to
\eqref{eq:r5} and \eqref{eq:r4}--\eqref{eq:odata}.

By \eqref{eq:psi-ap}, in order to obtain a leading order WKB type approximate solution
it suffices to determine $O(h^0)$ and $O(h^1)$ terms of $\phi^h$
in the limit $h \to 0$.
Letting $h=0$ in \eqref{eq:sys}, 
we formally obtain the $O(h^0)$ term $(b_0,\phi_0)$ which solves
\begin{equation}\tag{\ref{eq:lsys}}
	\left\{
	\begin{aligned}
		&\partial_t b_0 + \nabla \phi \cdot \nabla b_0 + \frac12 b_0 \Delta \phi_0 = 0,\\
		&\partial_t \phi_0 + \frac12 |\nabla \phi_0|^2 + \l t^{\g-2}(|y|^{-\g}*|b_0|^2) =0, \\
		& b_{0|\tau=0}=a_0, \quad \phi_{0|\tau=0}=0
	\end{aligned}
	\right.
\end{equation}
introduced in Section \ref{sec:intro}.
The difficulty of finding $h^1$-terms lies in the following two respects;
firstly, the equation \eqref{eq:sys} depends on $h$ through the term $i\frac{h}{2}\Delta a^h$;
and secondly the initial data of \eqref{eq:sys} is moving at a speed $h^{\AG}$.
In Section \ref{subsec:timeexpansion}, we give the time expansion of $(b_0,\phi_0)$
around $t=0$.
We will obtain an expansion of the form
\begin{align}\label{eq:tmpexp}
	b_0(\tau,y) & {} \asymp \sum_{j=0}^\I \tau^{\g j} a_j(y), &
	\phi_0(\tau,y) & {} \asymp  \sum_{j=1}^\I \tau^{\g j -1} \varphi_j(y)
\end{align}
 (Proposition \ref{prop:t-expansion}). 
This expansion is essential in handling the moving initial-data.

In Section \ref{subsec:expansion},
we finally determine $O(h^1)$ term in the case $\al \ge 1$.
What to show is the existence of the limits
\begin{align*}
	\frac{a^h(\tau) - b_0(\tau)}{h} & {} \to  b_{\mathrm{equ}}(\tau) , &
	\frac{\phi^h(\tau) - \phi_0(\tau)}{h} & {} \to \phi_{\mathrm{equ}}(\tau)
\end{align*}
as $h \to 0$.
A formal differentiation of \eqref{eq:sys} with respect to $h$
suggests that $(b_{\mathrm{equ}},\phi_{\mathrm{equ}})$ may solve the linearized system
\begin{equation}\tag{\ref{eq:tbtp1}}
	\(
	\begin{aligned}
		&\d_\tau b_{\mathrm{equ}} + \nabla \phi_{\mathrm{equ}}\cdot \nabla b_0
		+ \nabla \phi_0 \cdot \nabla b_{\mathrm{equ}}
		+  \frac{1}{2} b_{\mathrm{equ}} \Delta \phi_0 + \frac{1}{2} b_0 \Delta \phi_{\mathrm{equ}}
		= \frac{i}{2} \Delta b_0, \\
		&\d_\tau \phi_{\mathrm{equ}} + \nabla \phi_0 \cdot  \nabla \phi_{\mathrm{equ}}
		+ \l \tau^{\g-2} (|x|^{-\g}* 2 \Re \overline{b_0} b_{\mathrm{equ}} ) = 0.
	\end{aligned}
	\right.
\end{equation}
By means of \eqref{eq:tmpexp}, the following estimates hold at the initial time:
\begin{align*}
	\frac{a^h(h^{\AG}) - b_0(h^{\AG})}{h} & {} = -\frac{b_0(h^{\AG}) - a_0}{h} = O(h^{\frac{\al \g}{\g -\al} -1}), \\
		\frac{\phi^h(h^{\AG}) - \phi_0(h^{\AG})}{h} & {} = -\frac{\phi_0(h^{\AG}) }{h} = O(h^{\frac{\al (\g-1)}{\g -\al} -1}).
\end{align*}
Note that $\frac{\al \g}{\g -\al}>1$ for all $\al \ge 1$, however,
$\frac{\al (\g-1)}{\g -\al} >1$ if $\al >1$ and
$\frac{\al (\g-1)}{\g -\al} = 1$ if $\al =1$.
In particular, 
\[
	\frac{\phi^h(h^{\AG}) - \phi_0(h^{\AG})}{h} \to
	\begin{cases}
	0 & \text{ if } \al >1, \\
	- \varphi_1 & \text{ if } \al = 1,
	\end{cases}
\]
where $\varphi_1$ is defined in \eqref{eq:tmpexp}.
Therefore, the $O(h^1)$ term is described by $(b_{\mathrm{equ}},\phi_{\mathrm{equ}})$ solving
\eqref{eq:tbtp1} with \eqref{eq:tbtp2} if $\al >1$ and
with \eqref{eq:tbtp3} if $\al = 1$.
In Section \ref{sec:sscase}, we consider the case $\al < 1$.
In this case, the above powers $\frac{\al \g}{\g -\al}$
and $\frac{\al (\g-1)}{\g -\al}$ are less than one, in general.
Therefore, there appear several terms which is order less than $O(h^1)$
in the expansion of $(a^h,\phi^h)$.
We determine all these terms and obtain the asymptotic behavior of $(a^h,\phi^h)$
(see, Theorem \ref{thm:sscase}).

In Sections \ref{sec:proof} and \ref{sec:sscase}, we mainly work with $v^h=\nabla \phi^h$ instead of $\phi^h$ itself.
Note that, by means of Lemma \ref{lem:int}, it is easy to construct $\phi^h$ from $v^h$.

\subsection{Existence of phase-amplitude form solution}\label{subsec:pexistence}
According to the strategy in Section \ref{subsec:strategy},
we first show that the system \eqref{eq:sys} has a unique solution.
\begin{proposition}\label{prop:p-existence}
Let Assumption \ref{asmp:1} be satisfied.
Assume $0<\al<\g$.
Then, there exists $T>0$ independent of $h$ such that,
for all $h\in (0,1]$, there exists a unique solution $\psi^h \in C([h^{\AG},T+h^{\AG}];H^\I)$
to \eqref{eq:r5}. Moreover, $\psi^h$ is written as
\[
	\psi^h = a^h e^{i\frac{\phi^h}{h}},
\]
where 
\[
	a^h \in C([h^{\AG},T+h^{\AG}];H^\I) \cap C^\I((h^{\AG},T+h^{\AG}];H^\I)
\]
and 
\begin{align*}
	\phi^h \in & C([h^{\AG},T+h^{\AG}];Y^{\I}_{(n/\g,\I],(n/(\g+1),\I]}) \\
	&{} \cap C^\I((h^{\AG},T+h^{\AG}];Y^{\I}_{(n/\g,\I],(n/(\g+1),\I]}).
\end{align*}
Moreover, there exists a limit $(b_0,\phi_0):=(a^h,\phi^h)_{|h=0}$ belonging
the same function space as $(a^h,\phi^h)$ ($h>0$), and $(a^h,\phi^h)$
converges strongly to $(b_0,\phi_0)$ as $h\to 0$.
Furthermore, $(b_0,\phi_0)$ solves \eqref{eq:lsys}.
\end{proposition}
\begin{proof}
We set velocity $v^h = \nabla \phi^h$.
Then, the pair $(a^h,v^h)$ solves
\begin{equation}\label{eq:sysav1}
\left\{
  \begin{aligned}
	&\d_\tau a^h = Q_1(a^h,v^h) +i\frac{h}{2} \Delta a^h, \\
	&\d_\tau v^h = Q_2(v^h,v^h) + \tau^{\g-2} Q_3(a^h,a^h), 
  \end{aligned}
\right.
\end{equation}
\begin{align}\label{eq:sysav2}
	a^h_{|\tau=h^{\AG}} &= a_0, &
	v^h_{|\tau=h^{\AG}} &= 0,
\end{align}
where $Q_1$, $Q_2$, and $Q_3$ are defined by \eqref{eq:q-1}, \eqref{eq:q-2},
and \eqref{eq:q-3}, respectively.
To fix the initial time, we employ the time translation $\tau = t+h^{\AG}$.
Then, the equation is
\begin{equation}\label{eq:sysav3}
\left\{
  \begin{aligned}
	&\d_t a^h = Q_1(a^h,v^h) +i\frac{h}{2} \Delta a^h, \\
	&\d_t v^h = Q_2(v^h,v^h) + (t+h^{\AG})^{\g-2} Q_3(a^h,a^h), 
  \end{aligned}
\right.
\end{equation}
\begin{align}\label{eq:sysav4}
	a^h_{|t=0} &= a_0, &
	v^h_{|t=0} &= 0.
\end{align}
Now, the assumption $\g >1$ implies that $(t+h^{\AG})^{\g-2}$ is integrable over $(0,T^*)$ for some $T^*>0$
and its integral is uniformly bounded with respect to $h$.
Fix $s>n/2+1$.
Then, applying Proposition \ref{prop:existence-sys} with $c_1^h=c_2^h=1$, $r^h=h/2$, $f^h(t)=(t+h^{\AG})^{\g-2}$,
and $B_i^h\equiv W_i^h \equiv R_i^h \equiv 0$,
we obtain the existence time $0<T\le T^*$ independent of $\eps$ and the unique solution
\[
	(a^h,v^h) \in C([0,T]; H^s \times Y^{s+1}_{(n/(\g+1),\I],2})
\]
to \eqref{eq:sysav3}--\eqref{eq:sysav4}
such that $\tSobn{a^h}{s}^2 + \tSobn{\nabla v^h}{s}^2$ is bounded uniformly with respect to $t\in [0,T]$.
The upper bound depends only on $\tSobn{a_0}{s}$.
Notice that Assumption \ref{asmp:existence} is satisfied for all $s>n/2+1$ and
$q_0 \in ]n/(\g+1),n[$.

By the equation and the Hardy-Littlewood-Sobolev inequality,
$\partial_t v^h$ belongs to $C([0,T]; Y^{s}_{(n/(\g+1),\I],2})$.
Hence, we see from Lemma \ref{lem:int} that $\partial_t \phi^h \in C([0,T]; Y^{s+1}_{(n/\g,\I],(n/(\g+1),\I]})$.
Since $\phi^h_{|t=0} = 0$, we also have 
\[
	\phi^h \in C([0,T]; Y^{s+1}_{(n/\g,\I],(n/(\g+1),\I]}).
\]
Therefore, we have $\phi^h \to 0 $ as $|x|\to \I$.
Then, again applying Lemma \ref{lem:int} to $v^h$,
we conclude that $\phi^h \in C([0,T]; Y^{s+2}_{(n/\g,\I],(n/(\g+1),\I])}$.
Since the existence time $T$ is independent of $s$, we have
$a^h \in C([0,T];H^\I)$ and $\phi^h \in C([0,T];Y^{\I}_{(n/\g,\I],(n/(\g+1),\I]})$.
The bootstrap argument gives the $C^\I$ regularity with respect time.
The existence of the limit $(b_0,\phi_0)$ and the convergence
$(a^h,\phi^h)\to (b_0,\phi_0)$ as $h\to 0$ follow from the latter part of
Proposition \ref{prop:existence-sys} and Lemma \ref{lem:int}.
\end{proof}

\subsection{Time expansion of the limit solution near $\tau=0$}\label{subsec:timeexpansion}
By Proposition \ref{prop:p-existence}, 
the system \eqref{eq:sysav1}--\eqref{eq:sysav2} 
has a unique solution even if $h=0$.
We keep working with $v^h=\nabla \phi^h$ instead of $\phi^h$
Write $(b_0,w_0):=(a^h,v^h)_{|h=0}$.
The main difficulty of describing the asymptotic behavior of $(a^h,v^h)$ as $h \to 0$
comes from the fact that the initial data is given at $\tau=h^{\AG}$.
In order to handle this $h$-dependence of the initial time,
we give a time expansion of $(b_0,w_0) $ around $\tau=0$.

Note that $(b_0,w_0)$ solves
\begin{align}\label{eq:sysb01}
	\d_\tau b_0 &{}= Q_1(b_0,w_0), &
	\d_\tau w_0 &{}= Q_2(w_0 ,w_0) + \tau^{\g-2} Q_3(b_0,b_0), 
\end{align}
\begin{align}\label{eq:sysb02}
	b_{0|\tau=0} &= a_0, & w_{0|\tau=0} & = 0,
\end{align}
where the quadratic forms $Q_i$ are defined by \eqref{eq:q-1}--\eqref{eq:q-3}.
\begin{proposition}\label{prop:t-expansion}
Let $(b_0,w_0)=(b_0,\nabla \phi_0)$ be the unique solution to \eqref{eq:sysb01}--\eqref{eq:sysb02} defined by
Proposition \ref{prop:p-existence}.
Then, it holds that
\begin{align}\label{eq:t-exp1}
	b_0(\tau,y) & {} = \sum_{j=0}^J \tau^{\g j} a_j(y) + o(\tau^{\g J}) \IN H^\I, \\
	\label{eq:t-exp2}
	w_0(\tau,y) & {} =  \sum_{j=1}^J \tau^{\g j -1} v_j(y) + o(\tau^{\g J-1}) \IN Y^\I_{(n/(\g+1),\I],2}
\end{align}
as $\tau\to 0$ for all $J$, where $a_0$ is the initial data for $b_0$, $a_j$ and $v_j$ are defined by
\[
	a_j = \frac1{\g j} \sum_{k_1\ge 0, k_2 \ge 1, k_1+k_2=j}Q_1(a_{k_1},v_{k_2})
\]
for $j \ge 1$, $v_1 = Q_3(a_0,a_0)/(\g-1)$, and
\[
	v_j = \frac1{\g j -1} \left[ 
	\sum_{k_1\ge 1, k_2 \ge 1, k_1+k_2=j}Q_2(v_{k_1},v_{k_2})
	+ \sum_{k_1\ge 0, k_2 \ge 0, k_1+k_2=j-1}Q_3(a_{k_1},a_{k_2})
	\right]
\]
for $j \ge 2$ with the quadratic forms $Q_i$ defined by \eqref{eq:q-1}--\eqref{eq:q-3}.
Moreover, $\phi_0$ is expanded as
\begin{equation}\tag{\ref{eq:p-exp}}
	\phi_0(\tau,y)  =  \sum_{j=1}^J \tau^{\g j -1} \varphi_j(y) + o(\tau^{\g J-1})
	\IN Y^\I_{(n/\g,\I],(n/(\g+1),\I]}
\end{equation}
as $\tau \to 0$ for all $J \ge 1$, where $\varphi_j$ is given by $\varphi_1=\frac{\l}{1-\g}(|y|^{-\g}*|a_0|^2)$ and
\begin{multline*}
	\varphi_j = \frac1{1-\g j} \Bigg[ 
	\sum_{k_1\ge 1, k_2 \ge 1, k_1+k_2=j}\frac12 (\nabla \varphi_{k_1} \cdot \nabla \varphi_{k_2}) \\
	+ \sum_{k_1\ge 0, k_2 \ge 0, k_1+k_2=j-1}\l(|y|^{-\g}*(a_{k_1}\overline{a_{k_2}})
	\Bigg] .
\end{multline*}
\end{proposition}
\begin{remark}
Using the ``$\asymp$'' sign defined in Notation \ref{not:asymp},
the above three expansions \eqref{eq:t-exp1}, \eqref{eq:t-exp2},
and \eqref{eq:p-exp} can be written as
$b_0(\tau) \asymp \sum_{j=0}^J \tau^{\g j} a_j$ in $H^\I$, 
$w_0(\tau) \asymp  \sum_{j=1}^J \tau^{\g j -1} v_j$ in $Y^\I_{(n/(\g+1),\I],2}$,
and $\phi_0(\tau) \asymp  \sum_{j=1}^J \tau^{\g j -1} \varphi_j$
in $Y^\I_{(n/\g,\I],(n/(\g+1),\I]}$, respectively
\end{remark}
\begin{proof}
We first note that,
by the definitions of $Q_i$, it  follows for all $s>n/2+1$ that
\begin{align*}
	&\norm{Q_1(b,v)}_{H^s}  \le C_s \norm{b}_{H^{s+1}} \norm{v}_{Y^{s+1}_{(n/(\g+1),\I],2}}, \\
	&\norm{Q_2(v_1,v_2)}_{Y^{s}_{(n/(\g+1),\I],2}}  \le C_s \norm{v_1}_{Y^{s}_{(n/(\g+1),\I],2}}
	\norm{v_2}_{Y^{s+1}_{(n/(\g+1),\I],2}}, \\
	&\norm{Q_3(b_1,b_2)}_{Y^{s}_{(n/(\g+1),\I],2}}  \le C_s \norm{b_1}_{H^s} \norm{b_2}_{H^s}.
\end{align*}
Therefore, we see that $a_l$ and $v_l$ are bounded in $H^\I$ and $Y^\I_{(n/(\g+1),\I],2}$, respectively.
For simplicity, in this proof we denote $Y^s_{(n/(\g+1),\I],2}$ by $Y^s$.
Denote $\sum_{j=0}^l \tau^{\g j}a_j$ and $\sum_{j=1}^l \tau^{\g j -1} v_j$ by
$\widetilde{a}_l$ and $\widetilde{v}_l$, respectively. 
Then, it suffices to show that
\begin{align}
	&\norm{b_0 -\widetilde{a}_l}_{L^\I([0,\tau];H^\I)} = o(\tau^{\g l}) & & \forall l\ge 0, \label{eq:t-expb}\\
	&\norm{w_0 -\widetilde{v}_l}_{L^\I([0,\tau];Y^\I)} = o(\tau^{\g l -1}) & & \forall l \ge 1.
	\label{eq:t-expw}
\end{align}
\subsection*{Step 1}
Since $b_0 \in C([0,T];H^\I)$ and $b_0(0)=a_0 = \widetilde{a}_0$, \eqref{eq:t-expb} is trivial if $l=0$.
We show \eqref{eq:t-expw} for $l=1$.
By the second equation of \eqref{eq:sysb01}, it holds for $s>n/2+1$ that
\begin{align*}
	\norm{w_0(\tau)}_{Y^s}
	& {} \le C_1\int_0^\tau \norm{w_0(t)}_{Y^s} dt + C_2 \int_0^\tau t^{\g-2} dt,\\
	& {} \le C_1 \tau \norm{w_0}_{L^\I((0,\tau];Y^s)} + C_2^\prime \tau^{\g-1}.
\end{align*}
where $C_1$ depends on $s$ and $C([0,T];Y^\I)$ norm of $w_0$,
and $C_2$ depends on $s$ and $C([0,T];H^\I)$ norm of $b_0$.
The right hand side is monotone increasing in time, hence this gives
\[
	\norm{w_0}_{L^\I((0,\tau];Y^s)}
	\le C_1 \tau \norm{w_0}_{L^\I((0,t];Y^s)} + C_2^\prime \tau^{\g-1}.
\]
Choose $\tau$ so small that $C_1 \tau \le 1/2$. Then, we obtain
\[
	\norm{w_0}_{L^\I((0,\tau];Y^\I)} = O(\tau^{\g-1})
\]
since $s>n/2+1$ is arbitrary.
Again by the equation, it holds that
\[
	w_0 - \widetilde{v}_1 = \int_0^\tau Q_2(w_0,w_0) dt + \int_0^\tau t^{\g-2} (Q_3(b_0,b_0-a_0) + Q_3(b_0-a_0,a_0)) dt.
\]
Since $w_0$ is order $O(\tau^{\g-1})$ in $L^\I((0,\tau];Y^\I)$,
the first integral of the right hand side 
is order $O(\tau^{2\g-1})$ in $L^\I((0,\tau];Y^\I)$.
Similarly, the fact that $b_0-a_0$ is order $o(1)$ in $L^\I((0,\tau];H^\I)$ shows
the second integral is order $o(\tau^{\g-1})$ in $L^\I((0,\tau];Y^\I)$,
which proves \eqref{eq:t-expw} for $l=1$.
\subsection*{Step 2}
We prove \eqref{eq:t-expb} and \eqref{eq:t-expw} for large $l$ by induction.
By the definition of $a_j$, an explicit calculation shows
\begin{align*}
	\partial_\tau b_0 = {} & Q_1 (b_0,w_0) 
	=  Q_1(b_0,w_0-\widetilde{v}_1) + Q_1(b_0-\widetilde{a}_0, \tau^{\g-1}v_1)
		+ \partial_\tau (\tau^{\g}a_1 ) \\	
	= {} & Q_1 (b_0,w_0-\widetilde{v}_2) + Q_1(b_0 - \widetilde{a}_0,\tau^{2\g-1}v_2 ) + Q_1(b_0 -\widetilde{a}_1 ,\tau^{\g-1}v_1) \\
	& {} + \partial_\tau (\tau^{\g}a_1 + \tau^{2\g}a_2) =  \cdots  \\
	= {} & Q_1 ( b_0, w_0 - \widetilde{v}_l )
		+ \sum_{l_1 =0}^{l-1} Q_1( b_0 - \widetilde{a}_{l_1}, \tau^{\g (l-l_1) -1} v_{l-l_1}) 
		+ \partial_\tau \( \sum_{j=1}^l \tau^{\g j} a_j \).
\end{align*}
Similarly, it holds that
\begin{align*}
	\partial_\tau w_0 = {} & Q_2(w_0,w_0) + \tau^{\g-2} Q_3 (b_0,b_0) = \cdots  \\
	= {} &  Q_2 ( w_0, w_0 - \widetilde{v}_l ) + \sum_{l_1 =1}^{l} Q_2\( w_0 - \widetilde{v}_{l_1},
	\tau^{\g (l-l_1+1) -1} v_{l-l_1+1}\)  \\
	& {} + \tau^{\g-2}\(Q_3 ( b_0, b_0 - \widetilde{a}_l ) 
	+ \sum_{l_1 =0}^{l} Q_3( b_0 - \widetilde{a}_{l_1}, \tau^{\g (l-l_1) } a_{l-l_1}) \)\\
	& {} + \partial_\tau \( \sum_{j=1}^{l+1} \tau^{\g j -1} v_j \).
\end{align*}
Integrating these identities with respect to time, we obtain
\begin{equation}\label{eq:t-expbl}
	b_0 - \widetilde{a}_l = \int_0^\tau \(Q_1 ( b_0, w_0 - \widetilde{v}_l )
		+ \sum_{l_1 =0}^{l-1} Q_1( b_0 - \widetilde{a}_{l_1}, t^{\g (l-l_1) -1} v_{l-l_1}) \)dt
\end{equation}
and
\begin{align}\label{eq:t-expwl}
	w_0 - \widetilde{v}_{l+1} = {} & \int_0^\tau 
	\(Q_2 ( w_0, w_0 - \widetilde{v}_l ) + \sum_{l_1 =1}^{l} Q_2\( w_0 - \widetilde{v}_{l_1},
	t^{\g (l-l_1+1) -1} v_{l-l_1+1}\) \)dt  \\
	& {} + \int_0^\tau t^{\g-2}\(Q_3 ( b_0, b_0 - \widetilde{a}_l ) 
	+ \sum_{l_1 =0}^{l} Q_3( b_0 - \widetilde{a}_{l_1}, t^{\g (l-l_1) } a_{l-l_1}) \)dt. \nonumber
\end{align}
Now, let $L \ge 1$ be an integer.
If \eqref{eq:t-expb} holds for $l \le L-1$ and \eqref{eq:t-expw} holds for $l \le L$,
then we see that \eqref{eq:t-expbl} gives \eqref{eq:t-expb} with $l=L$.
On the other hand, if both \eqref{eq:t-expb} and \eqref{eq:t-expw} hold for $l \le L$,
then we obtain \eqref{eq:t-expw} with $l=L+1$ from \eqref{eq:t-expwl}.
\smallbreak

The expansion of $\phi_0$ is an immediate consequence of 
 the expansion of $w_0=\nabla \phi_0$.
Since 
\[
	Q_2(v_{k_1},v_{k_2}) + Q_2(v_{k_2},v_{k_1})
	= -\nabla (v_{k_1}\cdot v_{k_2}) = -\nabla \(\frac12 v_{k_1}\cdot v_{k_2}+ \frac12 v_{k_2}\cdot v_{k_1}\),
\]
we deduce from the definition $v_j$ that
\begin{multline*}
	\nabla \varphi_j = v_j =
	\frac1{1-\g j} \nabla\Bigg[ 
	\sum_{k_1\ge 1, k_2 \ge 1, k_1+k_2=j}\frac12 (\nabla \varphi_{k_1} \cdot \nabla \varphi_{k_2}) \\
	+ \sum_{k_1\ge 0, k_2 \ge 0, k_1+k_2=j-1}\l(|y|^{-\g}*(a_{k_1}\overline{a_{k_2}})
	\Bigg] .
\end{multline*}
By Lemma \ref{lem:int}, $\varphi_j$ belongs to $Y^\I_{(n/\g,\I],(n/(\g+1),\I]}$.
\end{proof}

\subsection{Asymptotic behavior of the phase-amplitude form solution}\label{subsec:expansion}
The following proposition completes the proof of the theorem.
\begin{proposition}
Let Assumption \ref{asmp:1} satisfied and $\al \ge 1$.
Let $T>0$ and $(a^h,v^h)$ be as in Proposition \ref{prop:p-existence}.
Let $(b_0,w_0):=(a^h,v^h)_{|h=0}$.
Then, there exists $(b_{\mathrm{equ}},w_{\mathrm{equ}})\in C([0,T];H^\I\times Y^\I_{(n/(\g+1),\I],2})$ such that
 the following asymptotics holds:
\begin{equation}\label{eq:tmp006}
\begin{aligned}
	a^h(\tau) &{}= b_0(\tau) + h b_{\mathrm{equ}}(\tau-h^{\AG}) + o(h) \IN C([h^{\AG},T];H^\I), \\
	v^h(\tau) &{}= w_0(\tau) + h w_{\mathrm{equ}}(\tau-h^{\AG}) + o(h) \IN C([h^{\AG},T];Y^\I_{(n/(\g+1),\I],2}).
\end{aligned}
\end{equation}
Moreover, $(b_{\mathrm{equ}},w_{\mathrm{equ}})$ solves
\begin{equation}\label{eq:tmp003}
\left\{
  \begin{aligned}
	\d_\tau b_{\mathrm{equ}} ={}& Q_1(b_0,w_{\mathrm{equ}}) + Q_1(b_{\mathrm{equ}},w_0)
	 +i\frac{1}{2}\Delta b_0 , \\
	\d_\tau w_{\mathrm{equ}} ={}&  Q_2(w_0,w_{\mathrm{equ}}) + Q_2(w_{\mathrm{equ}},w_0) + \tau^{\g-2} \(  Q_3(b_0,b_{\mathrm{equ}}) + Q_3(b_{\mathrm{equ}},b_0)\) 
  \end{aligned}
\right.
\end{equation}
with the data
\begin{align}\label{eq:tmp004}
	b_{\mathrm{equ}|\tau=0}&{}= 0,&
	w_{\mathrm{equ}|\tau=0}&{} =
	\begin{cases}
		0 & \text{ if } \al >1, \\
		-v_1 & \text{ if } \al =1,
	\end{cases}
\end{align}
where $v_1$ is defined in Proposition \ref{prop:t-expansion}.
\end{proposition}
\begin{remark}
Since
\begin{align*}
	b_{\mathrm{equ}}(\tau-h^{\AG}) &{}= b_{\mathrm{equ}}(\tau) + o(1), &
	w_{\mathrm{equ}}(\tau-h^{\AG}) &{}= w_{\mathrm{equ}}(\tau) + o(1)
\end{align*}
by continuity, \eqref{eq:tmp006} implies
\begin{align*}
	a^h(\tau) &{}= b_0(\tau) + h b_{\mathrm{equ}}(\tau) + o(h) \IN C([h^{\AG},T];H^\I), \\
	v^h(\tau) &{}= w_0(\tau) + h w_{\mathrm{equ}}(\tau) + o(h) \IN C([h^{\AG},T];Y^\I_{(n/(\g+1),\I],2}).
\end{align*}
From this asymptotics and the transforms \eqref{eq:spct} and \eqref{eq:psi-ap},
we immediately obtain the asymptotics \eqref{eq:asymptotics}.
\end{remark}
\begin{proof}
Let $(a^h,v^h)$ be the solution to \eqref{eq:sysav1}--\eqref{eq:sysav2}.
Let $(b_0,w_0)$ be the solution to \eqref{eq:sysb01}--\eqref{eq:sysb02}.
We put
\begin{align*}
	b^h(\tau,y) &{}:= \frac{a^h(\tau +h^{\AG},y) - b_0(\tau+h^{\AG},y)}h, \\
	w^h(\tau,y) &{}:= \frac{v^h(\tau +h^{\AG},y) - w_0(\tau+h^{\AG},y)}h.
\end{align*}
Then, $(b^h,w^h)$ solves
\begin{equation}\label{eq:tmp001}
\left\{
  \begin{aligned}
	\d_\tau b^h ={}& h Q_1(b^h, w^h)  + Q_1(b_0,w^h) + Q_1(b^h,w_0)
	 +i\frac{1}{2}\Delta b_0 + i\frac{h}{2} \Delta b^h, \\
	\d_\tau w^h ={}& h Q_2(w^h,w^h)   + Q_2(w_0,w^h) + Q_2(w^h,w_0) \\
	&{}  + (\tau+h^{\AG})^{\g-2} \( h Q_3(b^h,b^h) + Q_3(b_0,b^h) + Q_3(b^h,b_0)\), 
  \end{aligned}
\right.
\end{equation}
\begin{align}\label{eq:tmp002}
	b^h_{|\tau=0}&= \frac{b_{0|\tau=0} - b_0(h^{\AG})}h, &
	w^h_{|\tau=0}&= \frac{w_{0|\tau=0} - w_0(h^{\AG})}h, 
\end{align}
We apply Proposition \ref{prop:existence-sys} with these initial data and
$c_1^h = c_2^h=h$, $r^h=h/2$, $f^h(t)=(t+h^{\AG})^{\g-2}$, $B_1^h=B_2^h=b_0$,
$W_1^h=W_2^h=w_0$, $R_1^h=(i/2)\Delta b_0$, and $R_2^h =0$.
Note that the initial data \eqref{eq:tmp002} is uniformly bounded if $\al \ge 1$
since an application of \eqref{eq:t-exp1} and \eqref{eq:t-exp2} gives
\begin{align*}
	\frac{b_{0|\tau=0} - b_0(h^{\AG})}h = {}& O(h^{\frac{\g (\al-1)}{\g-\al} +\frac{\al}{\g -\al}})
	\IN H^{\I}, \\
	\frac{w_{0|\tau=0} - w_0(h^{\AG})}h = {}& O(h^{\frac{\g (\al-1)}{\g-\al}})
	\IN Y^\I_{(n/(\g+1),\I],2}.
\end{align*}
The term $R_1^h$ satisfies $\tSobn{R_1^h}{s} \le \tSobn{b_0}{s+2}/2$.
Therefore, if $s-2>n/2+1$, that is, if $s>n/2+3$ then
Proposition \ref{prop:existence-sys} provides the unique solution
$(b^h,w^h) \in C([0,T-h^{\AG}];H^{s-2} \times Y^{s-1}_{(n/(\g+1),\I],2}) $ for $h \in [0,1]$.
Moreover, $(b^h,w^h)$ converges to $(\ti{b},\ti{w}):=(b^h,w^h)_{|h=0}$ in 
$C([0,T-h^{\AG}];H^{s-4} \times Y^{s-3}_{(n/(\g+1),\I],2}) $.
It follows from \eqref{eq:t-exp2} that
$\lim_{h\to 0} w^h_{|\tau=0} = 0$ if $\al > 1$ and $\lim_{h\to 0} w^h_{|\tau=0} = -v_1$ if $\al = 1$.
Hence, $(b_{\mathrm{equ}},w_{\mathrm{equ}})$ solves \eqref{eq:tmp003}--\eqref{eq:tmp004}.
\end{proof}

\section{Supercritical caustic and Supercritical WKB case}\label{sec:sscase}
\subsection{Result}
In this section, we treat the case $\al <1<\g$.
As presented in Section \ref{subsec:strategy}, the asymptotic
behavior of the solution \eqref{eq:r4}--\eqref{eq:odata} boils down to
the asymptotic behavior of the solution to \eqref{eq:sys}.
By means of Lemma \ref{lem:int}, we work with $(a^h,v^h):=(a^h,\nabla \phi^h)$ which solves \eqref{eq:sysav1}--\eqref{eq:sysav2}.

The main difficulty lies in the fact that
the initial data \eqref{eq:sysav2} is moving at the speed $h^{\AG}$
(see Section \ref{subsec:summary2}).
From the expansion \eqref{eq:tmpexp} of $(b_0,w_0):=(a^h,v^h)_{|h=0}$,
we deduce that $(a^h,v^h)$ contains the terms of order
\[
	O(h^{\frac{\al(\g i -1)}{\g-\al}}) \quad \text{ and } \quad
	O(h^{\frac{j\al\g }{\g-\al}})
\]
for all $i,j \ge 0$.
Note that some of these orders are less than one if $\al < 1$.
This is the feature of the supercritical WKB case,
and the problem comes from this point.
Moreover,
the above terms interact each other and there appear all the terms whose
order is given by the finite combination of $h^{\frac{\al(\g i -1)}{\g-\al}}$
and $h^{\frac{j\al\g }{\g-\al}}$.
Thus, we see that $(a^h,v^h)$ contains all the terms whose order is written as
\[
	O(h^{\frac{\al (\g l_1 - l_2)}{\g -\al}}), \quad 0 \le l_2 \le l_1.
\]
For our purpose, we determine all these terms up to $O(h^1)$.
Therefore, it is natural to introduce a set $P$ defined by
\begin{equation}\label{def:P}
	P := \left\{ \frac{\al (\g l_1 - l_2)}{\g -\al}; 0 \le l_2 \le l_1, \quad 0 \le \frac{\al (\g l_1 - l_2)}{\g -\al} <1
	\right\}.
\end{equation}
Set $N:= \sharp P -1$, and number the elements of $P$ as $0=p_0<p_1<\cdots<p_{N} < 1$.
For any $p_{i_1},p_{i_2} \in P$, either
$p_{i_1} + p_{i_2} \in P$ or $p_{i_1}+p_{i_2} \ge 1$ holds.
For example, if $\g = \sqrt{3}$ and $\al = \sqrt{3}/4$ then,
$p_0=0$,
\begin{equation}\label{eq:sampleP1}
\begin{aligned}
	p_1 &{}=\frac{\sqrt{3}-1}{3} = \frac{\al(\g -1)}{\g-\al}, &
	p_2 &{}=\frac{2(\sqrt{3}-1)}{3} = \frac{\al(2\g -2)}{\g-\al}, \\
	p_3 &{}=\frac{\sqrt{3}}{3} = \frac{\al\g }{\g-\al}, &
	p_4 &{}=\sqrt{3}-1 = \frac{\al(3\g -3)}{\g-\al}, \\
	p_5 &{}=\frac{2\sqrt{3}-1}{3} = \frac{\al(2\g -1)}{\g-\al}, & &
\end{aligned}
\end{equation}
and $N=5$;
and if $\g=2$ and $\al = 1/3$, then  $p_0=0$,
\begin{equation}\label{eq:sampleP2}
\begin{aligned}
	p_1 &{}=\frac15 = \frac{\al(\g -1)}{\g-\al}, \qquad
	p_2 =\frac25 = \frac{\al(2\g -2)}{\g-\al}= \frac{\al\g}{\g-\al}, \\
	p_3 &{}=\frac35 = \frac{\al(3\g -3)}{\g-\al}= \frac{\al(2\g-1) }{\g-\al}, \\
	p_4 &{}=\frac45 = \frac{\al(4\g -4)}{\g-\al}= \frac{\al(3\g -2)}{\g-\al}
	= \frac{2\al\g}{\g-\al},
\end{aligned}
\end{equation}
and $N=4$.

To state the result, we also introduce several systems.
Let $Q_1$, $Q_2$, and $Q_3$ be quadratic forms defined in \eqref{eq:q-1},
\eqref{eq:q-2}, and \eqref{eq:q-3}, respectively.
Let $a_l$ and $v_l$ be sequences given in Proposition \ref{prop:t-expansion}.
Then, for any $0 \le i \le N$, we introduce
\begin{equation}\label{eq:tmp301}
	\left\{
	\begin{aligned}
		\partial_\tau b_i ={}& \sum_{p_j+p_k = p_i} Q_1(b_{p_j},w_{p_k}), \\
		\partial_\tau w_i ={}& \sum_{p_j+p_k = p_i} \(Q_2(w_{p_j},w_{p_k})+\tau^{\g-2}Q_3(b_{p_j},b_{p_k})\), 
 	\end{aligned}
 	\right.
\end{equation}
\begin{equation}\label{eq:tmp302}
\begin{aligned}
	b_i(0) ={}&
	\begin{cases}
	-a_l & \text{ if }\exists l \text{ such that } p_i=\frac{\al\g l}{\g-\al} , \\
	0 & \text{ otherwise},
	\end{cases}
	\\
	w_i(0) ={}&
	\begin{cases}
	-v_{l^\prime} & \text{ if } \exists l^\prime
	\text{ such that } p_i=\frac{\al\g l^\prime-\al}{\g-\al}, \\
	0 & \text{ otherwise}.
	\end{cases}
\end{aligned}
\end{equation}
We also introduce a system for $(b_{\mathrm{equ}},w_{\mathrm{equ}})$
\begin{equation}\label{eq:tmp304}
	\left\{
	\begin{aligned}
		\d_\tau b_{\mathrm{equ}} ={}& Q_1(b_0,w_{\mathrm{equ}}) + Q_1(b_{\mathrm{equ}},w_0) +
		\sum_{p_j+p_k = 1} Q_1(b_{p_j},w_{p_k}) + \frac{i}{2}\Delta b_0, \\
		\d_\tau w_{\mathrm{equ}} ={}& Q_2(w_0,w_{\mathrm{equ}}) + Q_2(w_{\mathrm{equ}},w_0) + 
		\tau^{\g-2}(Q_3(b_0,b_{\mathrm{equ}}) + Q_3(b_{\mathrm{equ}},b_0)) \\
		&{} + \sum_{p_j+p_k = 1} \(Q_2(w_{p_j},w_{p_k})+\tau^{\g-2}Q_3(b_{p_j},b_{p_k})\), \\
 	\end{aligned}
 	\right.
\end{equation}
\begin{equation}\label{eq:tmp305}
\begin{aligned}
	b_{\mathrm{equ}}(0) ={}&
	\begin{cases}
	-a_l & \text{ if }\exists l \text{ such that } 1=\frac{\al\g l}{\g-\al}, \\
	0 & \text{ otherwise},
	\end{cases}
	\\
	w_{\mathrm{equ}}(0) ={}&
	\begin{cases}
	-v_{l^\prime} & \text{ if }\exists l^\prime \text{ such that }1=\frac{\al\g l^\prime-\al}{\g-\al}, \\
	0 & \text{ otherwise},
	\end{cases}
\end{aligned}
\end{equation}
where $(b_0,w_0)$ is the solution of \eqref{eq:lsys}.
If there is no pair $(j,k)$ such that $p_j + p_k=1$, we let $\sum_{p_j+p_k = 1} \equiv 0$.
It may happen (see \eqref{eq:sampleP1}).

\begin{theorem}\label{thm:sscase}
Let assumption \ref{asmp:1} be satisfied. Assume $0<\al <1$.
Let $P$ be as in \eqref{def:P} and $N=\sharp P -1$.
Then, there exists an existence time $T>0$ independent of $\eps$.
There also exist
$(b_j,\phi_j) \in C([0,T];H^\I \times Y^\I_{(n/\g,\I],(n/(\g+1),\I]})$
($0 \le j \le N$)
such that  $(b_i,w_i):=(b_i,\nabla \phi_i)$ solves \eqref{eq:tmp301}--\eqref{eq:tmp302},
 and
$(b_{\mathrm{equ}}, \phi_{\mathrm{equ}}) \in C([0,T]; H^\I \times Y^\I_{(n/\g,\I],(n/(\g+1),\I]})$
such that $(b_{\mathrm{equ}},w_{\mathrm{equ}}):=(b_{\mathrm{equ}},\nabla \phi_{\mathrm{equ}})$ solves \eqref{eq:tmp304}--\eqref{eq:tmp305}.
Moreover, the followings hold:
\begin{enumerate}
\item $\phi_0(\tau) \asymp \sum_{j=1}^\I \tau^{\g j -1} \varphi_j$
(in the sense of \eqref{eq:p-exp}).
\item The solution $u^\eps$ to \eqref{eq:r4}--\eqref{eq:odata}
satisfies the following asymptotics for all $s \ge 0$:
\begin{equation}\label{eq:ssasymptotics}
	\sup_{t \in [0,1-T^{-1}\eps^{\al/\g}]} \Lebn{|J^\eps|^s \( u^\eps(t) e^{-i\Phi^\eps (t)}
	- \frac{1}{(1-t)^{n/2}} A^\eps(t) e^{i\frac{|\cdot|^2}{2\eps(t-1)}} \)}{2} \to 0
\end{equation}
as $\eps \to 0$ with
\begin{equation}\label{def:ssPhi}
	\Phi^\eps(t,x) = \eps^{\frac{\al}{\g}-1} \(
	\phi_0\(\frac{\eps^{\frac{\al}{\g}}}{1-t},\frac{x}{1-t}\) 
	+ \sum_{j=1}^N \eps^{(1-\frac{\al}{\g})p_j} \phi_j\(\frac{t\eps^{\frac{\al}{\g}}}{1-t},\frac{x}{1-t}\) \)
\end{equation}
and
\begin{equation}\label{def:ssA}
	A^\eps(t,x) = 
	b_0\(\frac{\eps^{\frac{\al}{\g}}}{1-t},\frac{x}{1-t} \)
	\exp \(i\phi_{\mathrm{equ}}\(\frac{\eps^{\frac{\al}{\g}}}{1-t},\frac{x}{1-t} \)\) .
\end{equation}
\end{enumerate}
\end{theorem}
\begin{remark}
In \eqref{def:ssPhi}, the time variable of $\phi_j$ ($j \ge 1$)
is not $\eps^{\frac{\al}{\g}}/(1-t)$ but 
\[
	\frac{t\eps^{\frac{\al}{\g}}}{1-t}=\frac{\eps^{\frac{\al}{\g}}}{1-t} - \eps^{\frac{\al}{\g}}.
\]	
Although this variable is not stable 
on the final layer $1-t = T^{-1} \eps^{\frac{\al}{\g}}$,
this choice is suitable when we work with the well-prepared data (see Section \ref{subsec:wpdata}).
Of course, the Taylor expansion
\[
	\phi_j\( \frac{t\eps^{\frac{\al}{\g}}}{1-t}\) = \sum_{k=0}^\I
	(-\eps^{\frac{\al}{\g}})^k (\partial_t^k \phi_j)\( \frac{\eps^{\frac{\al}{\g}}}{1-t}\)
\]
will exclude the variable ${t\eps^{\frac{\al}{\g}}}/(1-t)$
from $\Phi^\eps$, however, we do not pursue this point any more.
\end{remark}
\begin{remark}\label{rmk:type}
Let us classify the phase functions in \eqref{eq:ssasymptotics} according to
the notion in Section \ref{subsec:summary2}.
If $\phi_i(0)\not\equiv 0$ (resp. $b_i(0)\not\equiv 0$), that is, if
there exists a number $l \ge 1$ such that
$p_i=\frac{\al\g l-\al}{\g-\al}$
(resp. $p_i=\frac{\al\g l}{\g-\al}$),
then $\phi_i$ is the correction from phase (resp. correction from amplitude);
in particular $\phi_i=\phi_{\mathrm{pha},l}$ (resp. $\phi_i=\phi_{\mathrm{amp},l}$).
On the other hand, if $\phi_i(0)\equiv b_i(0) \equiv 0$ then
$\phi_i$ is the correction from interaction;
in particular $\phi_i=\phi_{\mathrm{int},l^\prime}$ for some $l^\prime$.
Notice that the summation in the system \eqref{eq:tmp301} is decomposed as
\[
	\sum_{p_j+p_k = p_i} = \sum_{(j,k)=(i,0),(0,i)}
	+ \sum_{p_j+p_k = p_i,jk\neq0}.
\]
The second sum is the interaction term, which is an external force.
When $\phi_i$ is the correction from interaction, it always has a nonzero interaction term.
Otherwise, $\phi_i\equiv0$ since the system for $\phi_i$
is posed with the zero initial condition.
There is a possibility that the correction from phase (resp. correction from amplitude)
has an interaction term
Indeed,
it happens if there is a triplet $j,k,l$ such that $p_i=\frac{\al\g l-\al}{\g-\al}=p_j+p_k$
(resp. $p_i=\frac{\al\g l}{\g-\al}=p_j+p_k$) and $jk\neq0$.
In this case, there is a resonance between the correction from phase
(resp. correction from amplitude) and the correction from interaction.
$\phi_{\mathrm{equ}}$ solving \eqref{eq:tmp304}--\eqref{eq:tmp305}
is the correction from equation with or without resonance:
If $\ti{\phi}(0)\not \equiv 0$ (resp. $\ti{b}(0)\not \equiv 0$ )
then there is a resonance with
the correction from phase (resp. correction from amplitude);
if $\sum_{p_j+p_k}\not\equiv0$ then there is a resonance with
the correction from interaction.
We note that the resonance among the correction from
phase, the correction from interaction, and the correction from equation
(resp. the correction from amplitude, the correction from interaction, and the correction from equation)
may happen, and that, however, 
the resonance between the correction from phase and the correction from amplitude
never happens because there is no pair $l,l^\prime$ such that
$\g l = \g l^\prime -1$ if $\g>1$.
\end{remark}
\begin{proof}
As presented in Section \ref{subsec:strategy}, the asymptotic
behavior of the solution \eqref{eq:r4}--\eqref{eq:odata} boils down to
the asymptotic behavior of the solution to 
\eqref{eq:sys}.
By means of Lemma \ref{lem:int}, we work with $(a^h,v^h):=(a^h,\nabla \phi^h)$ which solves \eqref{eq:sysav1}--\eqref{eq:sysav2}.

The existence of $(a^h,v^h)$ and the expansion of $\phi_0$ are already proven
in Propositions \ref{prop:p-existence} and \ref{prop:t-expansion}, respectively.
Let $P$ be as in \eqref{def:P}, $N=\sharp P -1$,
and $p_i\in P$ ($i=0,1,\dots,N$) be such that $\{ p_i\}_{i=0}^N=P$ and $p_i<p_{i+1}$.
It suffices to show that $(a^h,v^h)$ is expanded as
\begin{equation*}
	\begin{aligned}
		a^h(\tau+h^{\AG}) ={}& b_0(\tau+h^{\AG}) + \sum_{i=1}^N h^{p_i} b_i(\tau) + hb_{\mathrm{equ}}(\tau) + o(h), \\
		v^h(\tau+h^{\AG}) ={}& w_0(\tau+h^{\AG}) + \sum_{i=1}^N h^{p_i} w_i(\tau) + hw_{\mathrm{equ}}(\tau) + o(h).
	\end{aligned}
\end{equation*}
Plugging this and
\begin{align*}
	b_{\mathrm{equ}}(\tau)={}&b_{\mathrm{equ}}(\tau+h^{\AG}) + o(1), &
	w_{\mathrm{equ}}(\tau)={}&w_{\mathrm{equ}}(\tau+h^{\AG}) + o(1).
\end{align*}
to \eqref{eq:grenier} and \eqref{eq:spct},
we obtain \eqref{eq:ssasymptotics}.
\smallbreak

{\bf Step 1}.
We first prove by induction that
\begin{equation}\label{eq:tmp303}
	\begin{aligned}
		a^h(\tau+h^{\AG}) ={}& b_0(\tau+h^{\AG}) + \sum_{i=1}^{k} h^{p_i} b_i(\tau)  + o(h^{p_{k}}), \\
		v^h(\tau+h^{\AG}) ={}& w_0(\tau+h^{\AG}) + \sum_{i=1}^{k} h^{p_i} w_i(\tau)  + o(h^{p_{k}})
	\end{aligned}
\end{equation}
holds for $k=N$, where $(b_i,w_i)$ is a solution of 
\eqref{eq:tmp301}--\eqref{eq:tmp302}.
One verifies from Proposition \ref{prop:p-existence} that \eqref{eq:tmp303} holds if $k=0$.
We put $K \in [1, N]$.
We assume for induction that \eqref{eq:tmp303} holds for
 $k=K-1$, and put
\begin{equation*}
	\begin{aligned}
		b_{K}^h(\tau) ={}& h^{-p_{K}} \( a^h(\tau+h^{\AG}) - b_0(\tau+h^{\AG}) - \sum_{i=1}^{K-1} h^{p_i} b_i(\tau)\), \\
		w_{K}^h(\tau) ={}& h^{-p_{K}}\( v^h(\tau+h^{\AG}) - w_0(\tau+h^{\AG}) - \sum_{i=1}^{K-1} h^{p_i} w_i(\tau)\).
	\end{aligned}
\end{equation*}
By the equation for $(a^h,v^h)$, we see that $(b^h_K,w^h_K)$ solves
\begin{equation*}
\left\{
\begin{aligned}
	\d_\tau b^h_{K} ={}& h^{p_{K}}Q_1(b^h_{K},w^h_{K})
	 + Q_1(B_1^h,w^h_K) +  Q_1(b^h_K,W_1^k)  + R_1^h+ \frac{i}{2}h\Delta b^h_K, \\
	 \d_\tau w^h_{K} ={}& h^{p_{K}}Q_2(w^h_{K},w^h_{K})
	 + Q_2(W_2^h,w^h_K) +  Q_2(w^h_K,W_2^k)  \\
	&{} + (\tau+h^{\AG})^{\g-2} \(h^{p_{K}}Q_3(b^h_{K},b^h_{K})
	 + Q_3(B_2^h,b^h_K) +  Q_3(b^h_K,B_2^k) \) \\
	&{} + R_2^h , \\
		b_{K}^h(0) ={}& h^{-p_{K}} \( a_0 - b_0(h^{\AG}) - \sum_{i=1}^{K-1} h^{p_i} b_i(0)\), \\
		w_{K}^h(0) ={}& h^{-p_{K}}\( - w_0(h^{\AG}) - \sum_{i=1}^{K-1} h^{p_i} w_i(0)\),
\end{aligned}
\right.
\end{equation*}
where
\begin{align*}
	&B_1^h = B_2^h = \sum_{i=0}^{K-1} h^{p_i}b_i \to b_0(t), \\
	&W_1^h= W_2^h= \sum_{i=0}^{K-1} h^{p_i}w_i \to w_0(t),
\end{align*}
\begin{align*}
	R_1^h = {}& \sum_{p_l<K,p_k<K,p_l+p_k=p_{K}} Q_1 (b_{p_l},w_{p_k}) \\
	&{} + \sum_{p_l<K,p_k<K,p_l+p_k>p_{K}} h^{p_l+p_k-p_K}Q_1 (b_{p_l},w_{p_k}) \\
	&{} + \frac{ih^{1-p_K}}{2}\sum_{i=0}^{K-1}h^{p_i} \Delta b_i  \\
	\to {}& \sum_{p_l<K,p_k<K,p_l+p_k=p_{K}} Q_1 (b_{p_l},w_{p_k}),
\end{align*}
and
\begin{align*}
	R_2^h = {}& \sum_{p_l<K,p_k<K,p_l+p_k=p_{K}} \(Q_2 (w_{p_l},w_{p_k})
	+ (t+h^{\AG})^{\g-2}Q_3(b_{p_l},b_{p_k}) \) \\
	&{} + \sum_{p_l<K,p_k<K,p_l+p_k>p_{K}} h^{p_l+p_k-p_K}Q_2 (w_{p_l},w_{p_k}) \\
	&{} + \sum_{p_l<K,p_k<K,p_l+p_k>p_{K}} h^{p_l+p_k-p_K}(t+h^{\AG})^{\g-2}Q_3 (b_{p_l},b_{p_k}) \\
	\to {}& \sum_{p_l<K,p_k<K,p_l+p_k=p_{K}} \(Q_2 (w_{p_l},w_{p_k})
	+ t^{\g-2}Q_3(b_{p_l},b_{p_k}) \).
\end{align*}
Moreover, applying the time expansion of $(b_0,w_0)$,
we deduce that $(b_K^h(0),w_K^h(0))$ is uniformly bounded and that, as $h\to 0$,
\begin{align*}
	b_{K}^h(0) ={}& - h^{-p_{K}} \( b_0(h^{\AG})
	- \sum_{\{j;\frac{\al\g j }{\g -\al}<p_K\}} h^{\frac{\al\g j }{\g -\al}} a_j\), \\
	\to{}&
	\begin{cases}
	-a_{j^\prime} & \text{ if }\exists j^\prime \text{ such that } p_K = \frac{\al\g j^\prime }{\g -\al},   \\
	0 & \text{ otherwise},
	\end{cases} \\
	w_{K}^h(0) ={}& - h^{-p_{K}} \( w_0(h^{\AG})
	- \sum_{\{j;\frac{\al(\g j -1)}{\g -\al}<p_K\}} h^{\frac{\al(\g j -1)}{\g -\al}} v_j\), \\
	\to{}&
	\begin{cases}
	-v_{j^{\prime\prime}} & \text{ if }\exists j^{\prime\prime} \text{ such that } p_K = \frac{\al(\g j^{\prime\prime} -1)}{\g -\al},   \\
	0 & \text{ otherwise}.
	\end{cases}
\end{align*}
Note that either $b_K^h(0) \to 0$ or $w_K^h(0) \to 0$ holds for all $K$
since $\g >1$ implies there is no pair $(j,j^\prime)$
such that $\g j = \g j^\prime -1$.
Therefore, we apply Proposition \ref{prop:existence-sys} to obtain the solution $(b_K^h,w_K^h)$.
Put $(b_K,w_K):=(b_K^h,w_K^h)_{|h=0}$. 
Then, it solves \eqref{eq:tmp301}--\eqref{eq:tmp302}, and so \eqref{eq:tmp303}
holds for $k=K$.
By induction \eqref{eq:tmp303} holds for $k=N$.

{\bf Step 2}.
Mimicking the argument in Step 1,
we construct $(b_{\mathrm{equ}},w_{\mathrm{equ}})$ such that
\begin{equation*}
	\begin{aligned}
		a^h(\tau+h^{\AG}) ={}& b_0(\tau+h^{\AG}) + \sum_{i=1}^N h^{p_i} b_i(\tau) + hb_{\mathrm{equ}}(\tau) + o(h), \\
		v^h(\tau+h^{\AG}) ={}& w_0(\tau+h^{\AG}) + \sum_{i=1}^N h^{p_i} w_i(\tau) + hw_{\mathrm{equ}}(\tau) + o(h),
	\end{aligned}
\end{equation*}
holds.
Notice that $(b_{\mathrm{equ}},w_{\mathrm{equ}})$ solves
\begin{equation*}
	\begin{aligned}
		\d_\tau b_{\mathrm{equ}} ={}& Q_1(b_0,w_{\mathrm{equ}}) + Q_1(b_{\mathrm{equ}},w_0) +
		\sum_{p_j+p_k = 1} Q_1(b_{p_j},w_{p_k}) + \frac{i}{2}\Delta b_0, \\
		\d_\tau w_{\mathrm{equ}} ={}& Q_2(w_0,w_{\mathrm{equ}}) + Q_2(w_{\mathrm{equ}},w_0) + 
		\tau^{\g-2}(Q_3(b_0,b_{\mathrm{equ}}) + Q_3(b_{\mathrm{equ}},b_0)) \\
		&{} + \sum_{p_j+p_k = 1} \(Q_2(w_{p_j},w_{p_k})+\tau^{\g-2}Q_3(b_{p_j},b_{p_k})\), \\
 	\end{aligned}
\end{equation*}
\begin{align*}
	\ti{b}(0) ={}&
	\begin{cases}
	-a_l & \text{ if } 1=\frac{\al\g l}{\g-\al} \, \exists l, \\
	0 & \text{ otherwise},
	\end{cases}
	&
	\ti{w}(0) ={}&
	\begin{cases}
	-v_{l^\prime} & \text{ if } 1=\frac{\al\g l^\prime-\al}{\g-\al} \, \exists l^\prime, \\
	0 & \text{ otherwise}.
	\end{cases}
\end{align*}
\end{proof}
\begin{remark}
By a similar proof, we obtain higher order approximation.
Let $0<\al<\g$ and $\g>1$.
We modify the set $P$ defined by \eqref{def:P} as 
\[
	P^\prime :=\left\{l+\frac{\al\g m}{\g-\al} + \frac{\al(\g-1)n}{\g-\al}; \quad
	l,m,n \ge 0\right\},
\]
and number the elements of $P^\prime$ as
$0=p_0<p_1<\dots<p_k<\dots$.
Then, we have, for all $k$,
\begin{equation}\label{eq:higherexpansion}
	\begin{aligned}
		a^h(\tau+h^{\AG}) ={}& b_0(\tau+h^{\AG}) + \sum_{i=1}^k h^{p_i} b_i(\tau) + o(h^{p_k}), \\
		\phi^h(\tau+h^{\AG}) ={}& \phi_0(\tau+h^{\AG}) + \sum_{i=1}^k h^{p_i} \phi_i(\tau) + o(h^{p_k}).
	\end{aligned}
\end{equation}
Plugging this to \eqref{eq:grenier} and \eqref{eq:spct},
we obtain higher order approximation of the original solution.
It is important to note that the \eqref{eq:higherexpansion}
has the same form in the both $\al\ge1$ and $\al<1$ case.
When we concerned with higher order WKB-type approximation,
we need four kinds of correction terms even if $\al \ge 1$.
From this respect, the supercritical WKB case $\al <1$ 
can be characterized as the the special case $p_1<1$.
\end{remark}

\subsection{Well-prepared data and general data}\label{subsec:wpdata}
We  conclude this section with some remarks about the well-prepared data.
By the semiclassical conformal transform \eqref{eq:spct} and
Grenier's transform \eqref{eq:psi-ap}, the leading order WKB-type approximation of
the original solution $u^\eps$ to \eqref{eq:r4}--\eqref{eq:odata} is reduced to
the approximation of the solution $(a^h,\phi^h)$ to 
\begin{equation}\label{eq:sysa}
	\left\{
	\begin{aligned}
		&\partial_\tau a^h + \nabla \phi^h \cdot \nabla a^h + \frac12 a^h \Delta \phi^h = i\frac{h}{2}\Delta a^h,\\
		&\partial_\tau \phi^h + \frac12 |\nabla \phi^h|^2 + \l \tau^{\g-2}(|y|^{-\g}*|a^h|^2) =0
	\end{aligned}
	\right.
\end{equation}
with
\begin{align}\label{eq:sysb}
	a^h_{|\tau=h^{\frac{\al}{\g-\al}}}={}&a_0, &
	\phi^h_{|\tau=h^{\frac{\al}{\g-\al}}}={}&0,
\end{align}
up to order $O(h^1)$.
Note that \eqref{eq:sysa}--\eqref{eq:sysb} and \eqref{eq:sys} are the same.
As shown in Proposition \ref{prop:p-existence},
there exists a limit $(b_0,\phi_0):=(a^h,\phi^h)_{|h=0}$
which solves \eqref{eq:lsys}.
Now, we consider the distances
\begin{align*}
	d_a^h(t) :=&{} a^h(t) -b_0(t), &
	d_\phi^h(t) :=&{} \phi^h(t) - \phi_0(t).
\end{align*}
If these distances are order $o(h^1)$ then
we immediately obtain the WKB-type approximation $b_0\exp(i\eps^{\frac{\al}{\g}-1}\phi_0)$ of $u^\eps$
(recall that $h=\eps^{1-\frac{\al}{\g}}$).
However, unfortunately, the following two respects prevent us:
The first one is the $h$-dependence of the equation \eqref{eq:sysa},
and the second one is the $h$-dependence of the initial time \eqref{eq:sysb}.
The first problem is handled by employing the correction term $(b_{\mathrm{equ}},\phi_{\mathrm{equ}})$ solving \eqref{eq:tbtp1} and,
therefore we discuss about the initial data in the followings.
\smallbreak

The given initial data \eqref{eq:sysb} is written as
\begin{align*}
	d_a^h(h^{\AG}) :=&{} a_0 -b_0(h^{\AG}), &
	d_\phi^h(h^{\AG}) :=&{}  -\phi_0(h^{\AG}).
\end{align*}
The main difficulty in the case $\al \le 1$ is the fact that these terms
become larger than $O(h^1)$ as $h \to 0$.
The simplest way to overcome this difficulty is to modify
the initial data \eqref{eq:sysb} into
\begin{align}\label{eq:sysc}
	a^h_{|\tau=h^{\frac{\al}{\g-\al}}}={}&b_0(h^{\AG}), &
	\phi^h_{|\tau=h^{\frac{\al}{\g-\al}}}={}&\phi_0(h^{\AG}),
\end{align}
which ensures $d_a^h(h^{\AG}) \equiv d_\phi^h(h^{\AG}) \equiv 0$.
Note that \eqref{eq:sysa} with \eqref{eq:sysc} is the same as \eqref{eq:wpsys}.
Back to the transform \eqref{eq:spct}, this initial data corresponds to
the well-prepared data
\begin{equation}\tag{\ref{eq:wpdata}}
	u^\eps_{|t=0}(x) = b_0(\eps^{\frac{\al}{\g}},x)e^{-i\frac{|x|^2}{2 \eps}}
	\exp (i \eps^{\frac{\al}{\g} -1} \phi_0(\eps^{\frac{\al}{\g}},x)).
\end{equation}
This initial condition is rather natural in the supercritical case $\al<1$,
and the original initial condition in \eqref{eq:odata}
is a kind of constraint (see, Section \ref{subsec:summary2}).
If we use this well-prepared data, we do not have to consider any correction term
other than $(b_{\mathrm{equ}},\phi_{\mathrm{equ}})$ and, for all $0 < \al < \g$,
it holds that
\begin{align*}
	a^h(\tau) =&{} b_0(\tau) + h b_{\mathrm{equ}}(\tau) + o(h), \\
	\phi^h(\tau) =&{} \phi_0(\tau) + h \phi_{\mathrm{equ}}(\tau) + o(h)
\end{align*}
with $(b_0,\phi_0)$ and $(b_{\mathrm{equ}},\phi_{\mathrm{equ}})$
solving \eqref{eq:lsys} and \eqref{eq:tbtp1}--\eqref{eq:tbtp2}, respectively,
and so that the asymptotic behavior of the solution $u^\eps(t,x)$ to \eqref{eq:r4}--\eqref{eq:wpdata} is given by
\[
	b_0\(\frac{\eps^{\frac{\al}{\g}}}{1-t}, \frac{x}{1-t}\)
	\exp\( i \phi_{\mathrm{equ}}\(\frac{\eps^{\frac{\al}{\g}}}{1-t},\frac{x}{1-t}\)
	+ i\eps^{\frac{\al}{\g}-1} \phi_0 \(\frac{\eps^{\frac{\al}{\g}}}{1-t},\frac{x}{1-t}\)\).
\]
This approximate solution is the same one as in the case $\al >1$.
We still need $O(h^1)$-correction term 
$(b_{\mathrm{equ}},\phi_{\mathrm{equ}})$ because
it comes from the $h^1$-dependence of the equation which $(a^h,\phi^h)$ solves.

In Theorems \ref{thm:main} and \ref{thm:sscase}, we took another way.
By the expansion of $(b_0,\phi_0)$ around $\tau=0$,
there exist nonnegative integers
$k_1, k_2$ depending on $\al$ and $\g$ such that
\begin{align*}
	d_a^h(h^{\AG}) =&{} a_0 - b_0(h^{\AG}) = -\sum_{j=1}^{k_1} h^{\frac{\al\g j}{\g -\al}} a_j + o(h^1), \\
	d_\phi^h(h^{\AG}) =&{} -\phi _0(h^{\AG}) = -\sum_{j=1}^{k_2} h^{\frac{\al(\g j -1)}{\g -\al}} \varphi_j + o(h^1).
\end{align*}
We subtract the main part $-\sum_{j=1}^{k_1} h^{\frac{\al\g j}{\g -\al}} a_j$
and $-\sum_{j=1}^{k_2} h^{\frac{\al(\g j -1)}{\g -\al}} \varphi_j$
by constructing appropriate correction terms (correction from amplitude and correction from phase, respectively).
Indeed, if we let the correction terms $(b_j,\phi_j)$ ($0 \le j \le N$) and $(b_{\mathrm{equ}},\phi_{\mathrm{equ}})$
be defined as in Theorem \ref{thm:sscase} then, at the initial time $\tau=h^{\AG}$,
it holds that
\begin{align*}
	&\(a^h(\tau) - b_0(\tau) - \sum_{j=1}^N  h^{p_j} b_j(\tau-h^{\AG}) - hb_{\mathrm{equ}}(\tau)\)_{|\tau=h^{\AG}} \\
	&{}= -\(b_0(h^{\AG}) - \sum_{l \in \{l\ge 0;\frac{\al\g l}{\g -\al} \le 1\}} h^{\frac{\al\g l}{\g -\al}} a_l \)+ h\(b_{\mathrm{equ}}(0)-b_{\mathrm{equ}}(h^{\AG})\) \\
	&{}=o(h^1),
\end{align*}
\begin{align*}
	&\(\phi^h(\tau) - \phi_0(\tau) - \sum_{j=1}^N  h^{p_j} \phi_j(\tau-h^{\AG}) - h\phi_{\mathrm{equ}}(t)\)_{|\tau=h^{\AG}} \\
	&{}= -\(\phi_0(h^{\AG}) - \sum_{l \in \{l\ge1;\frac{\al(\g l-1)}{\g -\al} \le 1\}} h^{\frac{\al(\g l-1)}{\g -\al}} \varphi_l \)+ h\(\phi_{\mathrm{equ}}(0)-\phi_{\mathrm{equ}}(h^{\AG})\) \\
	&{}=o(h^1).
\end{align*}
It is important to note that the time variable of $(b_j,\phi_j)$ ($1 \le j \le N$)
is not $\tau$ but $\tau-h^{\AG}$. 
This choice is the key for the above cancellation.
By the transform \eqref{eq:spct}, the time variable of $\phi_j$ 
($1 \le j \le N$) in the definition \eqref{def:ssPhi} of $\Phi^\eps(t)$
should be given by
\[
	\frac{\eps^{\frac{\al}{\g}}}{1-t} - \eps^{\frac{\al}{\g}} =
	\frac{t\eps^{\frac{\al}{\g}}}{1-t}.
\]
These correction terms allow us to work with general data.

\section{Proofs of Theorems \ref{thm:main2} and \ref{thm:main3}}\label{sec:blowup}
Recall that the system we consider is
\begin{equation}\tag{\ref{eq:EPr}}
\left\{
\begin{aligned}
&\d_\tau (\rho r^{n-1}) + \partial_r (|\rho v r^{n-1}) = 0, \\
&\d_\tau v + v \partial_r v - \l \tau^{n-4} \d_r V_{\mathrm{p}} = 0, \\
&\d_r (r^{n-1}\partial_r V_{\mathrm{p}}) = \rho r^{n-1}\\
&\rho_{|\tau=0}=|a_0|^2, \quad v_{|\tau=0}=0,
\end{aligned}
\right.
\end{equation}
where $r=|x|$.
We introduce the ``mass'' $m$ and the ``mean mass'' $M$
\[
	m(\tau,r):=M(\tau,r)r^{n}:= r^{n-1}\d_r \Phi(\tau,r) = \int_0^r \rho(\tau,s) s^{n-1}ds.
\]
We also set $m_0(r):=m(0,r)$ and $M_0(r):=M(0,r)$.
Combining the first and the third equations of \eqref{eq:EPr}, we obtain
\begin{equation}\label{eq:math}
	\partial_\tau m + v \partial_r m = 0,
\end{equation}
where we have used $(\rho v r^{n-1})_{|r=0}=0$.
To solve this equation we also introduce the characteristic curve $X(\tau,R)$:
\begin{equation*}
	\frac{d X}{d \tau} = v(\tau,X(\tau,R)), \quad X(0,R) = R.
\end{equation*}
Denoting differentiation along this characteristic curve by $\prime:=d/d\tau$,
the mass equation \eqref{eq:math} and the second equation of \eqref{eq:EPr} yield
\begin{align}\label{eq:EPr2a}
 &m^\prime = 0,  \\
 \label{eq:EPr2b}
 &v^\prime = \l \tau^{n-4} \frac{m}{X^{n-1}},  \\
 \label{eq:EPr2c}
 &X^\prime = v.
\end{align}
We solve this system with the initial data
\[
	(X,m,v)_{|\tau=0} = (R,m_0(R),0),
\]
where $R \geq 0$ parameterizes the initial location.

By \eqref{eq:EPr2a}, the mass $m$ remains constant along the characteristics,
that is, $m(\tau,X(\tau,R)) = m_0(R)$.
Therefore, \eqref{eq:EPr2b} and \eqref{eq:EPr2c} yield
\begin{align}\label{eq:Xdd}
	X^{\prime\prime} &= \frac{\l m_0(R)\tau^{n-4}}{X^{n-1}},
	& X(0,R) & = R,
	& X^\prime(0,R) &= 0.
\end{align}
This equation is studied also in \cite{Tsu09}.

\begin{proof}[Proof of Theorem \ref{thm:main2}]
It suffices to show that $X(\tau,R)\le 0$ holds for some $R>0$ and $\tau \le T^*$.
This argument is similar to that for the blow-up for the nonlinear Schr\"odinger
equation by Glassey \cite{GlJMP}.
Since $\l<0$, we see from \eqref{eq:Xdd} that $X^{\prime\prime} \le 0$, and so that $X^{\prime} \le X^\prime(0)=0$ and $X \le X(0) = R$ for all $\tau \ge 0$.
Therefore, again by \eqref{eq:Xdd}, we verify that
\[
	X^{\prime\prime}(\tau,R)\le -\frac{|\l| m_0 \tau^{n-4}}{R^{n-1}}= - |\l| R M_0(R) \tau^{n-4}.
\]
Note that $M_0(R)>0$ for some $R$ provided $a_0$ is not identically zero.
Now, fix such $R$.
Integrating twice with respect to time, we obtain
\[
	X(\tau,R) \le R - \frac{|\l| RM_0(R)}{(n-2)(n-3)} \tau^{n-2},
\]
which yields $X(\tau,R) \le 0$ for large time.
The critical time is not greater than 
\[
	T^* = \(\frac{(n-2)(n-3)}{|\l| \sup_{R\ge 0}M_0(R)}\)^{\frac1{n-2}}
\]
since $R$ is arbitrary.
\end{proof}
\begin{proof}[Proof of Theorem \ref{thm:main3}]
In order to clarify the necessary and sufficient condition,
we repeat the proof in \cite{ELT-IUMJ}.
We first note that $\partial X/\partial R (0,R)=1$, and that
 the solution is global if and only if
 $\partial X/\partial R (t,R) > 0$ for all $\tau \ge 0$ and $R \ge 0$.

By \eqref{eq:Xdd}, we have
\[
	X^{\prime\prime} = \frac{\l m_0}{X^3}.
\]
We multiply this by $X^\prime$ and integrate in time to obtain
\[
	(X^\prime)^2 = 0^2 + \frac{\l m_0}{R^2} - \frac{\l m_0}{X^2}
	= \frac{\l m_0}{R^2} -XX^{\prime\prime},
\]
where we have used \eqref{eq:Xdd} again.
This yields
\[
	(X^2)^{\prime\prime} =2(X^\prime)^2 + 2XX^{\prime\prime} = \frac{2\l m_0}{R^2}.
\]
Then, integration twice gives
\[
	X^2 = R^2 + \frac{\l m_0}{R^2}t^2.
\]
Since $X \ge R > 0$ from \eqref{eq:Xdd} and $X^\prime(0)=0$, we see
\begin{equation}\label{eq:CharCurve}
	X(\tau,R) = \sqrt{R^2 + \frac{\l m_0}{R^2}\tau^2}.
\end{equation}
An explicit calculation shows that
\[
	\frac{\partial X}{\partial R}(\tau,R) = 
	\frac{ 1 + \l  (\frac12 |a_0|^2  - M_0 )\tau^2}
	{\sqrt{1 + \l M_0 \tau^2}},
\]
where $m_0 = M_0 R^4 = \int_0^R |a_0|^2s^3 ds$ by definition.
Note that $\lim_{R\to 0} M_0 = |a_0(0)|^2/4$ since $a_0$ is continuous.
Hence, $\partial X/\partial R (t,R) > 0$ holds for all $\tau \ge 0$ and $R \ge 0$
if and only if $\frac12 |a_0|^2  - M_0 \ge 0$ for all $R > 0$, that is,
\begin{equation}\label{eq:cond-global}
	|a_0(R)|^2 \ge \frac{2}{R^4}\int_0^R |a_0(s)|^2 s^3 ds
\end{equation}
for all $R > 0$.
Moreover, the critical time is given by
\[
	\tau_c = \(\frac{2}{\l \max_{r  > 0} \( 2M_0(r) - |a_0(r)|^2\)}\)^{\frac12}.
\]
The condition \eqref{eq:cond-global} can be written as
\[
	\partial_{R} \( \frac{m_0(R)}{R^2} \) \ge 0.
\]
We take $R_0$ so that $m_0(R_0) > 0$. Then, it holds only if
\[
	m_0(R) \ge \frac{m_0(R_0)}{R_0^2}R^2
\]
for all $R \ge R_0$,
which fails if $\lim_{R \to \I}m_0(R) < \I$, that is, if $a_0 \in L^2(\R^4)$.
\end{proof}
\begin{remark}\label{rem:sncond}
If $n=4$ then \eqref{eq:EPr} becomes an autonomous system.
Now we consider the classical solution of autonomous model
\begin{equation*}
\left\{
\begin{aligned}
&\partial_\tau (\rho r^{n-1}) + \partial_r (\rho v r^{n-1}) = 0, \\
&\partial_\tau v + v \partial_r v -\l \partial_r V_{\mathrm{p}} = 0, \\
& \partial_r (r^{n-1}\partial_r V_{\mathrm{p}}) = \rho r^{n-1},\\
&\rho_{|\tau=0}=\rho_0, \quad v_{|\tau=0}=v_0,
\end{aligned}
\right.
\end{equation*}
where $(\tau,r)\in \R_+\times \R_+$.
It is shown in \cite{MaRIMS} that, under the assumption that
$n\ge 3$, $\rho_0\in L^1((0,\I),r^{n-1}dr)$, $v_0(0)=0$,
and $v_0(r)\to0$ as $r\to\I$,
the corresponding solution is global if and only if $\l<0$ and
\[
	v_0(r) =\sqrt{\frac{2|\l|}{(n-2)r^{n-2}}\int_0^r \rho_0(s)s^{n-1}ds}. 
\]
\end{remark}
\begin{remark}\label{rem:indicator}
The function $\Gamma=\partial X/\partial R$ is called the indicator function.
As in the above proof, the solution blows up if and only if $\Gamma$
takes non-positive value.
Moreover, the solution is given by
\begin{align*}
	\rho(t,X(t,R)) =&{} \frac{\rho_0(R)R^{n-1}}{X^{n-1}(t,R) \Gamma(t,R)}, \\
	v(t,X(t,R)) =&{} \frac{d X}{d t} (t,R).
\end{align*}
Example \ref{ex:blowup} is easily checked by this form
since the characteristic curve $X$ is given explicitly by \eqref{eq:CharCurve}
in the case of $\l<0$ and $N=4$.
\end{remark}

\subsection*{Acknowledgments}
The author expresses his deep gratitude
to Professor R\'emi Carles for his reading an early version of the paper
and fruitful discussions in Montpellier.
Deep appreciation goes to Professor Yoshio Tsutsumi for his 
valuable advice and constant encouragement.
This research progressed during the author's stay in Orsay.
The author is grateful to people in department of mathematics, University of
Paris 11 for their kind hospitality.
This research is supported by JSPS fellow.

\providecommand{\bysame}{\leavevmode\hbox to3em{\hrulefill}\thinspace}
\providecommand{\href}[2]{#2}


\begin{thebibliography}{10}

\bibitem{AC-SP}
T.~Alazard and R.~Carles, \emph{Semi-classical limit of
  {S}chr\"odinger--{P}oisson equations in space dimension $n\ge 3$}, J.
  Differential Equations \textbf{233} (2007), no.~1, 241--275.

\bibitem{AC-ARMA}
\bysame, \emph{Supercritical geometric optics for nonlinear {S}chr\"odinger
  equations}, Arch. Ration. Mech. Anal. \textbf{194} (2009), no.~1, 315--347.

\bibitem{AC-GP}
\bysame, \emph{{WKB} analysis for the {G}ross-{P}itaevskii equation with
  non-trivial boundary conditions at infinity}, Ann. Inst. H. Poincare Anal.
  Non Lineaire \textbf{26} (2009), no.~3, 959--977.

\bibitem{CaIUMJ}
R.~Carles, \emph{Geometric optics with caustic crossing for some nonlinear
  {S}chr\"odinger equations}, Indiana Univ. Math. J. \textbf{49} (2000), no.~2,
  475--551.

\bibitem{CaCMP}
\bysame, \emph{Geometric optics and long range scattering for one-dimensional
  nonlinear {S}chr\"odinger equations}, Comm. Math. Phys. \textbf{220} (2001),
  no.~1, 41--67.

\bibitem{CaJHDE}
\bysame, \emph{Cascade of phase shifts for nonlinear {S}chr\"odinger
  equations}, J. Hyperbolic Differ. Equ. \textbf{4} (2007), no.~2, 207--231.

\bibitem{CaBook}
\bysame, \emph{Semi-classical analysis for nonlinear {S}chr\"odinger
  equations}, World Scientific Publishing Co. Pte. Ltd., Hackensack, NJ, 2008.

\bibitem{CL-FT}
R.~Carles and D.~Lannes, \emph{Focusing at a point with caustic crossing for a
  class of nonlinear equations}, 2nd France-Tunisia meeting (2003).

\bibitem{CM-AA}
R.~Carles and S.~Masaki, \emph{Semiclassical analysis for {H}artree equations},
  Asymptotic Analysis \textbf{58} (2008), no.~4, 211--227.

\bibitem{CMS-SIAM}
R.~Carles, N.~J. Mauser, and H.~P. Stimming, \emph{(semi)classical limit of the
  {H}artree equation with harmonic potential}, SIAM J. Appl. Math \textbf{66}
  (2005), no.~1, 29--56.

\bibitem{CR-CMP}
D.~Chiron and F.~Rousset, \emph{Geometric optics and boundary layers for
  {N}onlinear-{S}chr\"odinger {E}quations}, Comm. Math. Phys. \textbf{288}
  (2009), no.~2, 503--546.

\bibitem{ELT-IUMJ}
S.~Engelberg, H.~Liu, and E.~Tadmor, \emph{Critical thresholds in
  {E}uler-{P}oisson equations}, Indiana Univ. Math. J. \textbf{50} (2001),
  no.~Special Issue, 109--157, Dedicated to Professors Ciprian Foias and Roger
  Temam (Bloomington, IN, 2000).

\bibitem{Gallo}
C.~Gallo, \emph{Schr\"odinger group on {Z}hidkov spaces}, Adv. Differential
  Equations \textbf{9} (2004), no.~5-6, 509--538.

\bibitem{GLM-TJM}
I.~Gasser, C.-K. Lin, and P.~A. Markowich, \emph{A review of dispersive limits
  of (non)linear {S}chr\"odinger-type equations}, Taiwanese J. Math. \textbf{4}
  (2000), no.~4, 501--529.

\bibitem{PGEP}
P.~G{\'e}rard, \emph{Remarques sur l'analyse semi-classique de l'\'equation de
  {S}chr\"odinger non lin\'eaire}, S\'eminaire sur les \'Equations aux
  D\'eriv\'ees Partielles, 1992--1993, \'Ecole Polytech., Palaiseau, 1993,
  pp.~Exp.\ No.\ XIII, 13.

\bibitem{PGAHPANL}
\bysame, \emph{The {C}auchy problem for the {G}ross-{P}itaevskii equation},
  Ann. Inst. H. Poincar\'e Anal. Non Lin\'eaire \textbf{23} (2006), no.~5,
  765--779.

\bibitem{GlJMP}
R.~T. Glassey, \emph{On the blowing up of solutions to the {C}auchy problem for
  nonlinear {S}chr\"odinger equations}, J. Math. Phys. \textbf{18} (1977),
  no.~9, 1794--1797.

\bibitem{Grenier98}
E.~Grenier, \emph{Semiclassical limit of the nonlinear {S}chr\"odinger equation
  in small time}, Proc. Amer. Math. Soc. \textbf{126} (1998), no.~2, 523--530.

\bibitem{Hormander1}
L.~H{\"o}rmander, \emph{The analysis of linear partial differential operators.
  {I}}, second ed., Springer Study Edition, Springer-Verlag, Berlin, 1990,
  Distribution theory and Fourier analysis.

\bibitem{HK-WM}
J.~Hunter and J.~Keller, \emph{Caustics of nonlinear waves}, Wave motion
  \textbf{9} (1987), 429--443.

\bibitem{LL-EJDE}
H.~Li and C.-K. Lin, \emph{Semiclassical limit and well-posedness of nonlinear
  {S}chr\"odinger-{P}oisson systems}, Electron. J. Differential Equations
  (2003), No. 93, 17 pp. (electronic).

\bibitem{LZ-ARMA}
F.~Lin and P.~Zhang, \emph{Semiclassical limit of the {G}ross-{P}itaevskii
  equation in an exterior domain}, Arch. Ration. Mech. Anal. \textbf{179}
  (2006), no.~1, 79--107.

\bibitem{LT-MAA}
H.~Liu and E.~Tadmor, \emph{Semiclassical limit of the nonlinear
  {S}chr\"odinger-{P}oisson equation with subcritical initial data}, Methods
  Appl. Anal. \textbf{9} (2002), no.~4, 517--531.

\bibitem{MP-JJAM}
T.~Makino and B.~Perthame, \emph{Sur les solutions \`a sym\'etrie sph\'erique
  de l'\'equation d'{E}uler-{P}oisson pour l'\'evolution d'\'etoiles gazeuses},
  Japan J. Appl. Math. \textbf{7} (1990), no.~1, 165--170.

\bibitem{MaRIMS}
S.~Masaki, \emph{Remarks on global existence of classical solution to
  multi-dimensional compressible {E}uler-{P}oisson equations with geometrical
  symmetry}, RIMS Kokyuroku Bessatsu, to appear.

\bibitem{MaAHP}
\bysame, \emph{Semi-classical analysis for {H}artree equations in some
  supercritical cases}, Ann. Henri Poincar\'e \textbf{8} (2007), no.~6,
  1037--1069.

\bibitem{MaASPM}
\bysame, \emph{Semi-classical analysis of the {H}artree equation around and
  before the caustic}, Adv. Stud. in Pure Math. \textbf{47} (2007), no.~1,
  217--236.

\bibitem{PeJJAM}
B.~Perthame, \emph{Nonexistence of global solutions to {E}uler-{P}oisson
  equations for repulsive forces}, Japan J. Appl. Math. \textbf{7} (1990),
  no.~2, 363--367.

\bibitem{Tsu09}
I.~Tsukamoto, \emph{On asymptotic behavior of positive solutions of
  $x^{\prime\prime}=-t^{\alpha\lambda-2}x^{1+\alpha}$ with $\alpha<0$ and
  $-1<\lambda<0$}, preprint, 2009.

\bibitem{ZhSIAM}
P.~Zhang, \emph{Wigner measure and the semiclassical limit of
  {S}chr\"odinger-{P}oisson equations}, SIAM J. Math. Anal. \textbf{34} (2002),
  no.~3, 700--718 (electronic).

\bibitem{Zhidkov}
P.~E. Zhidkov, \emph{The {C}auchy problem for a nonlinear {S}chr\"odinger
  equation}, JINR Commun., P5-87-373, Dubna (1987), (in Russian).

\end{thebibliography}
\end{document}